\documentclass[11pt,oneside,english]{amsart}
\usepackage[T1]{fontenc}
\usepackage[latin9]{inputenc}
\usepackage{babel}
\usepackage{amsthm}
\usepackage{amssymb}
\usepackage{setspace}
\usepackage[unicode=true,
 bookmarks=true,bookmarksnumbered=false,bookmarksopen=false,
 breaklinks=false,pdfborder={0 0 1},backref=false,colorlinks=false]
 {hyperref}
\hypersetup{pdftitle={The structure of an isometric tuple},
 pdfauthor={Matthew Kennedy}}

\makeatletter
\numberwithin{equation}{section}
\numberwithin{figure}{section}
\theoremstyle{plain}
\newtheorem{thm}{\protect\theoremname}[section]
  \theoremstyle{definition}
  \newtheorem{defn}[thm]{\protect\definitionname}
  \theoremstyle{definition}
  \newtheorem{example}[thm]{\protect\examplename}
  \theoremstyle{plain}
  \newtheorem{prop}[thm]{\protect\propositionname}
  \theoremstyle{plain}
  \newtheorem{lem}[thm]{\protect\lemmaname}
  \theoremstyle{plain}
  \newtheorem{cor}[thm]{\protect\corollaryname}
  \theoremstyle{remark}
  \newtheorem*{acknowledgement*}{\protect\acknowledgementname}

\usepackage{microtype}

\makeatother

  \providecommand{\acknowledgementname}{Acknowledgement}
  \providecommand{\corollaryname}{Corollary}
  \providecommand{\definitionname}{Definition}
  \providecommand{\examplename}{Example}
  \providecommand{\lemmaname}{Lemma}
  \providecommand{\propositionname}{Proposition}
\providecommand{\theoremname}{Theorem}

\begin{document}

\title{The structure of an isometric tuple}

\author{Matthew Kennedy}

\address{School of Mathematics and Statistics\\
Carleton University\\
1125 Colonel By Drive\\
Ottawa, Ontario K1S 5B6\\
Canada}

\email{mkennedy@math.carleton.ca}
\begin{abstract}
An $n$-tuple of operators $(V_{1},\ldots,V_{n})$ acting on a Hilbert
space $H$ is said to be isometric if the operator $[V_{1}\ \cdots\ V_{n}]:H^{n}\to H$
is an isometry. We prove a decomposition for an isometric tuple of
operators that generalizes the classical Lebesgue-von Neumann-Wold
decomposition of an isometry into the direct sum of a unilateral shift,
an absolutely continuous unitary and a singular unitary. We show that,
as in the classical case, this decomposition determines the weakly
closed algebra and the von Neumann algebra generated by the tuple.
\end{abstract}


\thanks{Research partially supported by an NSERC Canada Graduate Scholarship.}

\maketitle

\section{Introduction}

This paper concerns the structure of an isometric tuple of operators,
an object that appears frequently in mathematics and mathematical
physics. From the perspective of an operator theorist, the notion
of an isometric tuple is a natural higher-dimensional generalization
of the notion of an isometry.

An $n$-tuple of operators $(V_{1},\ldots,V_{n})$ acting on a Hilbert
space $H$ is said to be \textbf{isometric} if the row operator $[V_{1}\ \cdots\ V_{n}]:H^{n}\to H$
is an isometry. This is equivalent to requiring that the operators
$V_{1},\ldots,V_{n}$ satisfy the algebraic relations 
\[
V_{i}^{*}V_{j}=\begin{cases}
I & \mbox{if }i=j,\\
0 & \mbox{if }i\ne j.
\end{cases}
\]
These relations are often referred to as the Cuntz relations.

The main result in this paper is a decomposition of an isometric tuple
that generalizes the classical Lebesgue-von Neumann-Wold decomposition
of an isometry into the direct sum of a unilateral shift, an absolutely
continuous unitary and a singular unitary. We show that, as in the
classical case, this decomposition determines the structure of the
weakly closed algebra and the von Neumann algebra generated by the
tuple.

The existence of a higher-dimensional Lebesgue-von Neumann-Wold decomposition
was conjectured by Davidson, Li and Pitts in \cite{DLP05}. They observed
that the measure-theoretic definition of an absolutely continuous
operator was equivalent to an operator-theoretic property of the functional
calculus for that operator. Since this property naturally extends
to the higher-dimensional setting, this allowed them to define the
notion an absolutely continuous isometric tuple.

The key technical result in this paper is a more effective operator-algebraic
characterization of an absolutely continuous isometric tuple. The
lack of such a characterization had been identified as the biggest
obstruction to establishing the conjecture in \cite{DLP05} (see also
\cite{DY08}). As we will see, the difficulty here can be attributed
to the lack of a higher-dimensional analogue of the spectral theorem.

In this paper, we overcome this difficulty by extending ideas from
the commutative theory of dual algebras to the noncommutative setting.
A similar approach was used in \cite{Ken11} to prove that certain
isometric tuples are hyperreflexive. In the present paper, the assumptions
on the isometric tuples we consider are much weaker, and the problem
is substantially more difficult. The idea to use this approach was
inspired by results of Bercovici in \cite{Ber98}.

In Section \ref{sec:motivation}, we review the Lebesgue-von Neumann-Wold
decomposition of a single isometry that is the motivation for our
results. In Section \ref{sec:background-preliminaries}, we provide
a brief review of the requisite background material on higher-dimensional
operator theory, and we introduce the notions of absolute continuity
and singularity. In Section \ref{sec:abs-cont}, we prove an operator-algebraic
characterization of an absolutely continuous isometric tuple. In Section
\ref{sec:sing-isom-tuple}, we prove an operator-algebraic characterization
of a singular isometric tuple. In Section \ref{sec:lebesgue-wold-decomp},
we prove the Lebesgue-von Neumann-Wold decomposition of an isometric
tuple, and we obtain some consequences of this result.

The present exposition was inspired by the perspective of Muhly and
Solel in \cite{MS10}, which appeared shortly after the first version
of this paper. They consider the notion of absolute continuity in
a  more general setting.

\section{Motivation\label{sec:motivation}}

The structure of a single isometry $V$ is well understood. By the
Wold decomposition of an isometry, $V$ can be decomposed as
\[
V=V_{u}\oplus U,
\]
where $V_{u}$ is a unilateral shift of some multiplicity, and $U$
is a unitary. By the Lebesgue decomposition of a measure applied to
the spectral measure of $U$, we can decompose $U$ as
\[
U=V_{a}\oplus V_{s},
\]
where $V_{a}$ is an absolutely continuous unitary and $V_{s}$ is
a singular unitary, in the sense that their spectral measures are
absolutely continuous and singular respectively with respect to Lebesgue
measure. This allows us to further decompose $V$ as
\[
V=V_{u}\oplus V_{a}\oplus V_{s}.
\]
We will refer to this as the \textbf{Lebesgue-von Neumann-Wold decomposition}
of an isometry.

It will be convenient to consider the above notions of absolute continuity
and singularity from a different perspective. Let $A(\mathbb{D})$
denote the classical disk algebra of analytic functions on the complex
unit disk $\mathbb{D}$ with continuous extension to the boundary.
An isometry $V$ induces a contractive representation of $A(\mathbb{D})$,
namely the $A(\mathbb{D})$ functional calculus for $V$, given by
\[
f\to f(V),\quad f\in A(\mathbb{D}).
\]
Recall that the algebra $A(\mathbb{D})$ is a weak-{*} dense subalgebra
of the algebra $H^{\infty}$ of bounded analytic functions on the
complex unit disk. In certain cases, the representation of $A(\mathbb{D})$
induced by $V$ is actually the restriction to $A(\mathbb{D})$ of
a weak-{*} continuous representation of $H^{\infty}$, namely the
$H^{\infty}$ functional calculus for $V$, given by
\[
f\to f(V),\quad f\in H^{\infty}.
\]
It follows from Theorem III.2.1 and Theorem III.2.3 of \cite{SF70}
that this occurs if and only if $V_{s}=0$ in the Lebesgue-von Neumann-Wold
decomposition of $V$. This motivates the following definitions.
\begin{defn}
\label{def:class-abs-cont}Let $V$ be an isometry. We will say that
$V$ is \textbf{absolutely continuous} if the representation of $A(\mathbb{D})$
induced by $V$ extends to a weak\nobreakdash-{*} continuous representation
of $H^{\infty}$. If $V$ has no absolutely continuous restriction
to an invariant subspace, then we will say that $V$ is \textbf{singular}. 
\end{defn}
The importance of the Lebesgue-von Neumann-Wold decomposition of an
isometry $V$ is that it determines the structure of the weakly closed
algebra $\mathrm{W}(V)$ and the von Neumann algebra $\mathrm{W}^{*}(V)$
generated by $V$. Recall that $\mathrm{W}(V)$ is the weak closure
of the polynomials in $V$, and $\mathrm{W}^{*}(V)$ is the weak closure
of the polynomials in $V$ and $V^{*}$.

Let $\alpha$ denote the multiplicity of $V_{u}$ as a unilateral
shift, and let $\mu_{a}$ and $\mu_{s}$ be scalar measures equivalent
to the spectral measures of $V_{a}$ and $V_{s}$ respectively. Since
a unilateral shift of multiplicity one is irreducible, $\mathrm{W}^{*}(V)$
is given by
\[
\mathrm{\mathrm{W}}^{*}(V)\simeq B(\ell^{2})^{\alpha}\oplus L^{\infty}(V_{a})\oplus L^{\infty}(\mu_{s})(V_{s}).
\]
It was established by Wermer in \cite{Wer52} that $\mathrm{W}(V)$
can be self-adjoint, depending on $\alpha$ and $\mu_{a}$. If $\alpha\ne0$
or if Lebesgue measure is absolutely continuous with respect to $\mu_{a}$,
then $\mathrm{W}(V)$ is given by
\[
\mathrm{W}(V)\simeq H^{\infty}(V_{u}\oplus V_{a})\oplus L^{\infty}(\mu_{s})(V_{s}).
\]
Otherwise, if neither of these conditions holds, then $\mathrm{W}(V)=\mathrm{W}^{*}(V)$.

The following example shows that it is possible for the weakly closed
algebra generated by an absolutely continuous isometry to be self-adjoint.
We will see later that there is no higher-dimensional analogue of
this phenomenon.
\begin{example}
\label{exa:reductive-abs-cont-uni}

Let $U$ denote the operator of multiplication by the coordinate function
on $L^{2}(\mathbb{T},m)$, where $m$ denotes Lebesgue measure. Let
$m_{1}$ and $m_{2}$ denote Lebesgue measure on the upper and lower
half of the unit circle respectively, and let $U_{1}$and $U_{2}$
denote the operator of multiplication by the coordinate function on
$L^{2}(\mathbb{T},m_{1})$ and $L^{2}(\mathbb{T},m_{2})$ respectively.

Since the spectral measure of $U\simeq U_{1}\oplus U_{2}$ is equivalent
to Lebesgue measure, $U$ is absolutely continuous. Thus $U_{1}$
and $U_{2}$ are also absolutely continuous. From above,
\[
\mathrm{W}^{*}(U)\simeq L^{\infty}(U),\qquad\mathrm{W}(U)\simeq H^{\infty}(U).
\]
However, since Lebesgue measure is not absolutely continuous with
respect to $m_{1}$ or $m_{2}$, 
\[
\mathrm{W}(U_{i})=\mathrm{W}^{*}(U_{i})=L^{\infty}(U_{i}),\quad i=1,2.
\]
In particular, the weakly closed algebras $\mathrm{W}(U_{1})$ and
$\mathrm{W}(U_{2})$ generated by $U_{1}$ and $U_{2}$ respectively
are self-adjoint.
\end{example}

\section{Background and preliminaries\label{sec:background-preliminaries}}

\subsection{\label{sub:noncomm-funct-alg}The noncommutative function algebras}

The noncommutative Hardy space $F_{n}^{2}$ is defined to be the full
Fock-Hilbert space over $\mathbb{C}^{n}$, i.e.
\[
F_{n}^{2}=\oplus_{k=0}^{\infty}(\mathbb{C}^{n})^{\otimes k},
\]
where we will write $\xi_{\varnothing}$ to denote the vacuum vector,
so that $(\mathbb{C}^{n})^{\otimes0}=\mathbb{C}\xi_{\varnothing}$.
Let $\xi_{1},\ldots,\xi_{n}$ be an orthonormal basis of $\mathbb{C}^{n}$
and let $\mathbb{F}_{n}^{*}$ denote the unital free semigroup on
$n$ generators $\{1,\ldots,n\}$ with unit $\varnothing$. For a
word $w=w_{1}\cdots w_{k}$ in $\mathbb{F}_{n}^{*}$, it will be convenient
to write $\xi_{w}=\xi_{w_{1}}\otimes\cdots\otimes\xi_{w_{k}}$. We
can identify $F_{n}^{2}$ with the set of power series in $n$ noncommuting
variables $\xi_{1},\ldots,\xi_{n}$ with square-summable coefficients,
i.e.
\[
F_{n}^{2}=\left\{ \sum_{w\in\mathbb{F}_{n}^{*}}a_{w}\xi_{w}:\sum_{w\in\mathbb{F}_{n}^{*}}|a_{w}|^{2}<\infty\right\} .
\]
In particular, we can identify the noncommutative Hardy space $F_{1}^{2}$
with the classical Hardy space $H^{2}$ of analytic functions having
power series expansions with square-summable coefficients.

The left multiplication operators $L_{1},\ldots,L_{n}$ are defined
on $F_{n}^{2}$ by
\[
L_{i}\xi_{w}=\xi_{i}\otimes\xi_{w}=\xi_{iw},\quad w\in\mathbb{F}_{n}^{*}.
\]
It is clear that the $n$-tuple $L=(L_{1},\ldots,L_{n})$ is isometric.
We will call it the \textbf{unilateral $n$-shift} since, for $n=1$,
$L_{1}$ can be identified with the unilateral shift on $H^{2}$.
For a word $w=w_{1}\cdots w_{k}$ in $\mathbb{F}_{n}^{*}$, it will
be convenient to write $L_{w}=L_{w_{1}}\cdots L_{w_{k}}$.

The \textbf{noncommutative disk algebra} $\mathcal{A}_{n}$ is the
norm closed unital algebra generated by $L_{1},\ldots,L_{n}$ and
the \textbf{noncommutative analytic Toeplitz algebra} $\mathcal{L}_{n}$
is the weakly closed unital algebra generated by $L_{1},\ldots,L_{n}$.
These algebras were introduced by Popescu in \cite{Pop96}, and have
subsequently been studied by a number of authors (see for example
\cite{DP98} and \cite{DP99}).

The noncommutative disk algebra $\mathcal{A}_{n}$ and the noncommutative
analytic Toeplitz algebra $\mathcal{L}_{n}$ are higher-dimensional
analogues of the classical disk algebra $A(\mathbb{D})$ and the classical
algebra $H^{\infty}$ of bounded analytic functions. In particular,
the algebra $\mathcal{A}_{n}$ is a proper weak-{*} dense subalgebra
of the algebra $\mathcal{L}_{n}$. If we agree to identify functions
in $H^{\infty}$ with the corresponding multiplication operators on
$H^{2}$, then we can identify $A(\mathbb{D})$ with $\mathcal{A}_{1}$
and $H^{\infty}$ with $\mathcal{L}_{1}$.

As in the classical case, an element $A$ in $\mathcal{L}_{n}$ is
uniquely determined by its Fourier series
\[
A\sim\sum_{w\in\mathbb{F}_{n}^{*}}a_{w}L_{w},
\]
where $a_{w}=(A\xi_{\varnothing},\xi_{w})$ for $w$ in $\mathbb{F}_{n}^{*}$.
The Cesaro sums of this series converge strongly to $A$, and it is
often useful heuristically to work directly with this representation.

We will also need to work with the right multiplication operators
$R_{1},\ldots,R_{n}$ defined on $F_{n}^{2}$ by
\[
R_{i}\xi_{w}=\xi_{w}\otimes\xi_{i}=\xi_{wi},\quad w\in\mathbb{F}_{n}^{*}.
\]
The $n$-tuple $R=(R_{1},\ldots,R_{n})$ is unitarily equivalent to
$L=(L_{1},\ldots,L_{n})$. The unitary equivalence is implemented
by the ``unitary flip'' on $F_{n}^{2}$ that, for a word $w_{1}\cdots w_{k}$
in $\mathbb{F}_{n}^{*}$, takes $\xi_{w_{1}\cdots w_{k}}$ to $\xi_{w_{k}\cdots w_{1}}$.
We will let $\mathcal{R}_{n}$ denote the weakly closed algebra generated
by $R_{1},\ldots,R_{n}$.

\subsection{Free semigroup algebras}

Let $V=(V_{1},\ldots,V_{n})$ be an isometric $n$-tuple. The weakly
closed unital algebra $\mathrm{W}(V)$ generated by $V_{1},\ldots,V_{n}$
is called the \textbf{free semigroup algebra} generated by $V$. As
in Section \ref{sub:noncomm-funct-alg}, for a word $w=w_{1}\cdots w_{k}$
in the unital free semigroup $\mathbb{F}_{n}^{*}$, it will be convenient
to write $V_{w}=V_{w_{1}}\cdots V_{w_{k}}$.
\begin{example}
The noncommutative analytic Toeplitz algebra $\mathcal{L}_{n}$ introduced
in Section \ref{sub:noncomm-funct-alg} is a fundamental example of
a free semigroup algebra. We will see that it plays an important role
in the general theory of free semigroup algebras.
\end{example}
The study of free semigroup algebras was initiated by Davidson and
Pitts in \cite{DP99}. They observed that information about the unitary
invariants of an isometric tuple can be detected in the algebraic
structure of the free semigroup algebra it generates, and used this
fact to classify a large family of representations of the Cuntz algebra.
Free semigroup algebras have subsequently received a great deal of
interest (see for example \cite{Dav01}).

It was shown in \cite{DP98} that $\mathcal{L}_{n}$ has a great deal
of structure that is analogous to the analytic structure of $H^{\infty}$.
This motivates the following definition.
\begin{defn}
\label{def:analytic-isom-tuple}An isometric $n$-tuple $V=(V_{1},\ldots,V_{n})$
is said to be \textbf{analytic} if the free semigroup algebra generated
by $V$ is isomorphic to the noncommutative analytic Toeplitz algebra
$\mathcal{L}_{n}$.
\end{defn}

The notion of analyticity is of central importance in the theory of
free semigroup algebras. This is apparent from the work of Davidson,
Katsoulis and Pitts in \cite{DKP01}. They proved the following general
structure theorem.
\begin{thm}
[Structure theorem for free semigroup algebras]\label{thm:struct-thm-fsa}Let
$\mathcal{V}=\mathrm{W}(V)$ be a free semigroup algebra. Then there
is a projection $P$ in $\mathcal{V}$ with range invariant under
$\mathcal{V}$ such that
\begin{enumerate}
\item if $P\ne0,$ then the restriction of $\mathcal{V}$ to the range of
$P$ is an analytic free semigroup algebra,
\item the compression of $\mathcal{V}$ to the range of $P^{\perp}$ is
a von Neumann algebra,
\item $\mathcal{V}=P\mathcal{V}P+(\mathrm{W}^{*}(V))P^{\perp}$.
\end{enumerate}
\end{thm}

The analytic structure of a free semigroup algebra reveals itself
in the form of wandering vectors. Let $V=(V_{1},\ldots,V_{n})$ be
an isometric $n$-tuple acting on a Hilbert space $H$. A vector $x$
in $H$ is said to be \textbf{wandering} for $V$ if the set of vectors
$\{V_{w}x:w\in\mathbb{F}_{n}^{*}\}$ is orthonormal. In this case
we will also say that $x$ is wandering for the free semigroup algebra
generated by $V$.

The existence of wandering vectors for an analytic free semigroup
algebra was established in \cite{Ken11}, settling a conjecture first
made in \cite{DKP01} (see also \cite{DLP05} and \cite{DY08}). Examples
show that the structure of an analytic free semigroup algebra can
be quite complicated, making this result far from obvious.

\subsection{Dilation theory\label{sub:dilation-theory}}

Recall that an operator $T$ is said to be contractive if $\|T\|\leq1$.
An $n$-tuple of operators $T=(T_{1},\ldots,T_{n})$ acting on a Hilbert
space $H$ is said to be \textbf{contractive} if the row operator
$[T_{1}\ \cdots\ T_{n}]:H^{n}\to H$ is contractive.

Sz.-Nagy showed that every contractive operator $T$ acting on a Hilbert
space $H$ has a unique minimal dilation to an isometry $V$, acting
on a bigger Hilbert space $K$ (see for example \cite{SF70}). This
means that $H\subseteq K$, $H$ is cyclic for $V$ and
\[
T^{k}=P_{H}V^{k}\mid_{H},\quad k\geq1.
\]

Sz.-Nagy's dilation theorem was generalized in the work of Bunce,
Frazho and Popescu in \cite{Bun84}, \cite{Fra82} and \cite{Pop89a}
respectively. They showed that every contractive $n$-tuple of operators
$T=(T_{1},\ldots,T_{n})$ acting on a Hilbert space $H$ has a unique
minimal dilation to an isometric $n$-tuple $V=(V_{1},\ldots,V_{n})$,
acting on a bigger Hilbert space $K$. This means that $H\subseteq K$,
$H$ is cyclic for $V_{1},\ldots,V_{n}$ and 
\[
P_{H}V_{i_{1}}\cdots V_{i_{k}}\mid_{H}=T_{i_{1}}\cdots T_{i_{k}},\quad i_{1},\ldots,i_{k}\in\{1,\ldots,n\}\ \mbox{and\ }k\geq1.
\]

\subsection{The Wold decomposition}

The classical Wold decomposition decomposes a single isometry into
the direct sum of a unilateral shift of some multiplicity and a unitary.
In order to state the Wold decomposition of an isometric tuple, we
need to generalize these notions.

In Section \ref{sub:noncomm-funct-alg}, we introduced the unilateral
$n$-shift $L=(L_{1},\ldots,L_{n})$, and we saw that it is the natural
higher-dimensional generalization of the classical unilateral shift.
An isometric $n$-tuple is said to be a \textbf{unilateral shift of
multiplicity $\alpha$} if it is unitarily equivalent to the ampliation
$L^{(\alpha)}=(L_{1}^{(\alpha)},\ldots,L_{n}^{(\alpha)})$, for some
positive integer $\alpha$.

The higher-dimensional generalization of a unitary is based on the
fact that a unitary is the same thing as a surjective isometry. An
$n$-tuple of operators $U=(U_{1},\ldots,U_{n})$ is said to be \textbf{unitary}
if the operator $[U_{1}\ \cdots\ U_{n}]:H^{n}\to H$ is a surjective
isometry. This is equivalent to requiring that the operators $U_{1},\ldots,U_{n}$
satisfy
\[
\sum_{i=1}^{n}U_{i}U_{i}^{*}=I.
\]
Note that a unilateral shift is not unitary. This is because the ``vacuum''
vector $\xi_{\varnothing}$ in $F_{n}^{2}$ is not contained in the
range of the unilateral $n$-shift $L=(L_{1},\ldots,L_{n})$.

In \cite{DP99}, Davidson and Pitts studied a family of ``atomic''
isometric tuples that arise from certain infinite directed trees.
As the following example shows, this family contains a large number
of unitary tuples.
\begin{example}
Fix an infinite directed $n$-ary tree $B$ with vertex set $V$ such
that every vertex has a parent. For a vertex $v$ in $V$, let $c_{i}(v)$
denote the $i$-th child of $v$. Let $H=\ell^{2}(V)$, so that the
set $\{e_{v}:v\in V\}$ is an orthonormal basis for $H$. Define operators
$S_{1},\ldots,S_{n}$ on $H$ by
\[
S_{i}e_{v}=e_{c_{i}(v)},\quad1\leq i\leq n.
\]
It's clear that $S_{1},\ldots,S_{n}$ are isometries, and the fact
that $B$ is an infinite directed $n$-ary tree implies that the range
of $S_{i}$ and the range of $S_{j}$ are orthogonal for $i\ne j$.
Thus $S=(S_{1},\ldots,S_{n})$ is an isometric $n$-tuple. The fact
that every vertex has a parent implies that every basis vector is
in the range of some $S_{i}$. Thus $S$ is a unitary $n$-tuple.
\end{example}

Let $V=(V_{1},\ldots,V_{n})$ be an arbitrary isometric $n$-tuple.
If $V$ is unitary, then the $\mathrm{C}^{*}$-algebra $\mathrm{C}^{*}(V_{1},\ldots,V_{n})$
generated by $V$ is isomorphic to the Cuntz algebra $\mathcal{O}_{n}$.
Otherwise, it is isomorphic to the extended Cuntz algebra $\mathcal{E}_{n}$,
the extension of the compacts by $\mathcal{O}_{n}$. Since the only
irreducible {*}-representation of the compacts is the identity representation,
and since $\mathcal{O}_{n}$ is simple, a {*}-representation of $\mathcal{E}_{n}$
can be decomposed into a multiple of the identity representation and
a representation of $\mathcal{O}_{n}$. The Wold decomposition of
an isometric $n$-tuple, which was proved by Popescu in \cite{Pop89a},
can be obtained as a consequence of these $\mathrm{C}^{*}$-algebraic
facts, based on the observation that the $\mathrm{C}^{*}$-algebra
generated by a unilateral $n$-shift is isomorphic to $\mathcal{E}_{n}$.

\begin{prop}
[The Wold decomposition]\label{pro:wold-decomp}Let $V=(V_{1},\ldots,V_{n})$
be an isometric $n$-tuple. Then we can decompose $V$ as
\[
V=V_{u}\oplus U,
\]
where $V_{u}$ is a unilateral $n$-shift and $U$ is a unitary $n$-tuple.
\end{prop}

\subsection{Absolutely continuous and singular isometric tuples}

As in the classical case, an isometric $n$-tuple $V=(V_{1},\ldots,V_{n})$
induces a contractive representation of the noncommutative disk algebra
$\mathcal{A}_{n}$, called the $\mathcal{A}_{n}$ functional calculus
for $V$, determined by 
\[
L_{i_{1}}\cdots L_{i_{k}}\to V_{i_{1}}\cdots V_{i_{k}},\quad i_{1},\ldots,i_{k}\in\{1,\ldots,n\}\mbox{ and }k\geq1.
\]
This is a consequence of Popescu's generalization of von Neumann's
inequality in \cite{Pop91}.

Recall from Section \ref{sub:noncomm-funct-alg} that $\mathcal{A}_{n}$
is a proper weak-{*} dense subalgebra of the noncommutative analytic
Toeplitz algebra $\mathcal{L}_{n}$. The following definition is the
natural generalization of Definition \ref{def:class-abs-cont}.
\begin{defn}
\label{def:mult-abs-cont}Let $V=(V_{1},\ldots,V_{n})$ be an isometric
$n$-tuple. We will say that $V$ is \textbf{absolutely continuous}
if the representation of $\mathcal{A}_{n}$ induced by $V$ is the
restriction to $\mathcal{A}_{n}$ of a weak-{*} continuous representation
of $\mathcal{L}_{n}$. We will say that $V$ is \textbf{singular}
if $V$ has no absolutely continuous restriction to an invariant subspace.
\end{defn}

It is clear from Definition \ref{def:analytic-isom-tuple} and Definition
\ref{def:mult-abs-cont} that an analytic isometric tuple is absolutely
continuous. In order to obtain the Lebesgue-von Neumann-Wold decomposition
of an isometric tuple, we will prove the converse result that an absolutely
continuous isometric tuple is analytic.

\section{Absolutely continuous isometric tuples\label{sec:abs-cont}}

The main result in this section is an operator-algebraic characterization
of an absolutely continuous isometric tuple. Specifically, we will
show that for $n\geq2$, every absolutely continuous isometric $n$-tuple
is analytic.

For $n\geq2$, fix an absolutely continuous isometric $n$-tuple $S=(S_{1},\ldots,S_{n})$
acting on a Hilbert space $H$. Let $\Phi$ denote the corresponding
representation of the noncommutative disk algebra $\mathcal{A}_{n}$,
given by
\[
\Phi(L_{w})=S_{w},\quad w\in\mathbb{F}_{n}^{*}.
\]
Since $S$ is absolutely continuous, $\Phi$ extends to a representation
of $\mathcal{L}_{n}$ that is weak-{*} continuous.

It was shown in Corollary 1.2 of \cite{DY08} that $\Phi$ is actually
a completely isometric isomorphism and a weak-{*} homeomorphism from
$\mathcal{L}_{n}$ to the weak-{*} closed algebra generated by $S_{1},\ldots,S_{n}$.
This is equivalent to the fact that an infinite ampliation of $S$
is an analytic isometric tuple. Evidently, it is much more difficult
to show that $S$ is analytic. As an explanation, we offer the aphorism
that things are generally much nicer in the presence of infinite multiplicity.

Showing that $S$ is analytic amounts to showing that the free semigroup
algebra (i.e. the weakly closed algebra) $\mathrm{W}(S)$ generated
by $S_{1},\ldots,S_{n}$ is isomorphic to the noncommutative analytic
Toeplitz algebra $\mathcal{L}_{n}$. Since we know from above that
the weak-{*} closed algebra generated by $S_{1},\ldots,S_{n}$ is
isomorphic to $\mathcal{L}_{n}$, our strategy will be to show that
this algebra  is actually equal to $\mathrm{W}(S)$.

\subsection{\label{sub:abs-cont-dual-op-sys}The noncommutative Toeplitz operators}

Let $\mathcal{S}$ denote the weak-{*} closed algebra generated by
$S_{1},\ldots,S_{n}$. The map $\Phi$ introduced at the beginning
of this section is a completely isometric isomorphism and a weak-{*}
homeomorphism from $\mathcal{L}_{n}$ to $\mathcal{S}$. It will be
useful for what follows to extend $\Phi$ even further. Let $\mathcal{M}_{n}$
denote the weak-{*} closure of the operator system $\mathcal{L}_{n}+\mathcal{L}_{n}^{*}$.
We will call the elements of $\mathcal{M}_{n}$ the \textbf{noncommutative
Toeplitz operators}, because they are a natural higher-dimensional
generalization of the classical Toeplitz operators.

The noncommutative Toeplitz operators were introduced by Popescu in
\cite{Pop89b}. It was shown in Corollary 1.3 of \cite{Pop09} that
$A$ belongs to $\mathcal{M}_{n}$ if and only if
\[
R_{i}^{*}AR_{j}=\begin{cases}
A & \mbox{if }i=j,\\
0 & \mbox{otherwise},
\end{cases}
\]
where $R_{1},\ldots,R_{n}$ are the right multiplication operators
introduced in Section \ref{sub:noncomm-funct-alg}. A short proof
of this fact was also given in Lemma 3.2 of \cite{Ken11}. It follows
from this characterization that $\mathcal{M}_{n}$ is weakly closed.

Let $\mathcal{T}$ denote the weak-{*} closure of the operator system
$\mathcal{S}+\mathcal{S}^{*}$. The proof of the following proposition
is nearly identical to the proof of Theorem 3.6 of \cite{Ken11}.
\begin{prop}
Let $S=(S_{1},\ldots,S_{n})$ be an absolutely continuous isometric
$n$-tuple. The representation $\Phi$ of $\mathcal{L}_{n}$ induced
by $S$ extends to a completely isometric and weak-{*} homeomorphic
{*}-map from $\mathcal{M}_{n}$ to $\mathcal{T}$.
\end{prop}

We will need to exploit the fact that $\mathcal{M}_{n}$ and $\mathcal{T}$
are dual spaces. Let $\mathcal{T}_{*}$ denote the predual of $\mathcal{T}$,
i.e. the set of weak-{*} continuous linear functionals on $\mathcal{T}$.
Similarly, let $\mathcal{M}_{n*}$ denote the predual of $\mathcal{M}_{n}$.
Basic functional analysis implies that the inverse map $\Phi^{-1}$
is the dual of an isometric isomorphism $\phi$ from $\mathcal{M}_{n*}$
to $\mathcal{T}_{*}$. Moreover, since $\Phi^{-1}$ is isometric,
so is $\phi$.

We can identify the predual of $B(F_{n}^{2})$, i.e. the set of weak-{*}
continuous linear functionals on $B(F_{n}^{2})$, with the set of
trace class operators $C^{1}(F_{n}^{2})$ on $F_{n}^{2}$, where $K$
in $C^{1}(F_{n}^{2})$ corresponds to the linear functional
\[
(T,K)=\operatorname{tr}(TK),\quad T\in B(F_{n}^{2}).
\]
If we let $(\mathcal{M}_{n})_{\perp}$ denote the preannihilator of
$\mathcal{M}_{n}$, i.e. 
\[
(\mathcal{M}_{n})_{\perp}=\{K\in C^{1}(F_{n}^{2}):\operatorname{tr}(AK)=0,\quad\forall A\in\mathcal{M}_{n}\},
\]
then we can identify the predual $(\mathcal{M}_{n})_{*}$ with the
quotient space $C^{1}(F_{n}^{2})/(\mathcal{M}_{n})_{\perp}$. Similarly,
we can identify the predual $\mathcal{T}_{*}$ with the quotient space
$C^{1}(H)/\mathcal{T}_{\perp}$.

For $\xi$ and $\eta$ in $F_{n}^{2}$, it will be convenient to let
$[\xi\otimes\eta]_{\mathcal{M}_{n}}$ denote the weak-{*} continuous
linear functional on $\mathcal{M}_{n}$ given by
\[
(A,[\xi\otimes\eta]_{\mathcal{M}_{n}})=(A\xi,\eta),\quad A\in\mathcal{M}_{n}.
\]
In other words, $[\xi\otimes\eta]_{\mathcal{M}_{n}}$ denotes the
equivalence class of the rank one tensor $x\otimes y$ in $(\mathcal{M}_{n})_{*}$.
Similarly, for $x$ and $y$ in $H$, let $[x\otimes y]_{\mathcal{T}}$
denote the weak-{*} continuous linear functional on $\mathcal{T}$
given by
\[
(T,[x\otimes y]_{\mathcal{T}})=(Tx,y),\quad T\in\mathcal{T}.
\]

\subsection{Intertwining operators}

An operator $X:F_{n}^{2}\to H$ is said to \textbf{intertwine} the
isometric $n$-tuple $S=(S_{1},\ldots,S_{n})$ and the unilateral
$n$-shift $L=(L_{1},\ldots,L_{n})$ if it satisfies
\[
XL_{i}=S_{i}X,\quad1\leq i\leq n.
\]
Observe that if $X$ intertwines $S$ and $L$, then the operator
$JX^{*}XJ$ is a noncommutative Toeplitz operator, where $J$ is the
unitary flip introduced in Section \ref{sub:noncomm-funct-alg}. Indeed,
using the fact that $JR_{i}=L_{i}J$ for $1\leq i\leq n$, we compute
\begin{eqnarray*}
R_{i}^{*}JX^{*}XJR_{j} & = & JL_{i}^{*}X^{*}XL_{j}J\\
 & = & JX^{*}S_{i}^{*}S_{j}XJ\\
 & = & \begin{cases}
JX^{*}XJ & \mbox{if }i=j,\\
0 & \mbox{otherwise}.
\end{cases}
\end{eqnarray*}
Since $S$ is absolutely continuous, it follows from Theorem 2.7 of
\cite{DLP05} that every vector $x$ in $H$ is in the range of an
operator that intertwines $S$ and $L$.

\subsection{\label{sub:dual-alg-theory-prelim}Dual algebra theory}

Recall that to prove the isometric $n$-tuple $S=(S_{1},\ldots,S_{n})$
is analytic, our strategy is to show that the weak-{*} closed algebra
$\mathcal{S}=\mathrm{W}^{*}(S_{1},\ldots,S_{n})$ is actually equal
to the weakly closed algebra $\mathrm{W}(S_{1},\ldots,S_{n})$. This
amounts to showing that $\mathcal{S}$ is already weakly closed. However,
instead of working directly with $\mathcal{S}$, it will be necessary
to work with the operator system $\mathcal{T}$. In fact, we will
need to consider the general structure of the predual of $\mathcal{T}$.

In Section \ref{sub:abs-cont-dual-op-sys}, we saw that an element
in the predual $\mathcal{T}_{*}$ of the operator system $\mathcal{T}$
can be identified with an equivalence class of trace class operators.
We will show that $\mathcal{T}$ satisfies a very powerful predual
``factorization'' property, in the sense that the equivalence class
of an element in the predual $\mathcal{T}_{*}$ always contains ``nice''
representatives. We will see that $\mathcal{S}$ inherits this property
from $\mathcal{T}$, and that this will imply the desired result.

The idea of studying factorization in the predual of an operator algebra
is the central idea in dual algebra theory, which has been applied
with great success to a number of problems in the commutative setting
(see for example \cite{BFP85}). As we will see, many of the factorization
properties that were introduced in the commutative setting make sense
even in the present noncommutative setting.

\begin{defn}
A weak-{*} closed subspace $\mathcal{A}$ of operators acting on a
Hilbert space $H$ is said to have property $\mathbb{A}_{1}(1)$ if,
given a weak-{*} continuous linear functional $\tau$ on $\mathcal{A}$
with $\|\tau\|\leq1$ and $\epsilon>0$, there are vectors $x$ and
$y$ in $H$ such that $\|x\|\leq(1+\epsilon)^{1/2}$, $\|y\|\leq(1+\epsilon)^{1/2}$
and $\tau=[x\otimes y]_{\mathcal{A}}$.
\end{defn}

If a weak-{*} closed subspace of $\mathcal{B}(H)$ has property $\mathbb{A}_{1}(1)$,
then the equivalence class of any weak-{*} continuous linear functional
on the subspace contains an operator of rank one. Note that in this
case, every weak-{*} continuous linear functional on the subspace
is actually weakly continuous. It was shown in \cite{DP99} that $\mathcal{L}_{n}$
has property $\mathbb{A}_{1}(1)$, and the same proof also shows that
$\mathcal{M}_{n}$ has property $\mathbb{A}_{1}(1)$.

Of course, the main difficulty with a predual factorization property
like property $\mathbb{A}_{1}(1)$ is that it is often extremely difficult
to show that it holds. The next factorization property turns out to
be much stronger than property $\mathbb{A}_{1}(1)$, but it is sometimes
easier to show that it holds due to its approximate nature.

\begin{defn}
A weak-{*} closed subspace $\mathcal{A}$ of operators acting on a
Hilbert space $H$ is said to have property $\mathcal{X}_{0,1}$ if,
given a weak-{*} continuous linear functional $\tau$ on $\mathcal{A}$
with $\|\tau\|\leq1$, $z_{1},...,z_{q}$ in $H$ and $\epsilon>0$,
there are vectors $x$ and $y$ in $H$ such that
\begin{enumerate}
\item $\|x\|\leq1$ and $\|y\|\leq1$,
\item $\|[x\otimes z_{j}]_{\mathcal{A}}\|<\epsilon$ and $\|[z_{j}\otimes y]_{\mathcal{A}}\|<\epsilon$
for $1\leq j\leq q$,
\item $\|\tau-[x\otimes y]_{\mathcal{A}}\|<\epsilon$.
\end{enumerate}
\end{defn}

It's easy to see that the infinite ampliation of a weak-{*} closed
subspace of $\mathcal{B}(H)$ has property $\mathcal{X}_{0,1}$. Thus,
intuitively, a weak-{*} closed subspace of $\mathcal{B}(H)$ that
has property $\mathcal{X}_{0,1}$ can be thought of as having ``approximately
infinite'' multiplicity. It was shown in \cite{BFP85} that property
$\mathcal{X}_{0,1}$ implies property $\mathbb{A}_{1}(1)$.

We will show that $\mathcal{T}$ has property $\mathcal{X}_{0,1}$.
Since this property is inherited by weak-{*} closed subspaces, it
will follow that $\mathcal{S}$ has property $\mathcal{X}_{0,1}$,
and hence that $\mathcal{S}$ has property $\mathbb{A}_{1}(1)$. It
is easy to show that any weak-{*} closed subspace of operators with
property $\mathbb{A}_{1}(1)$ is weakly closed (see for example Proposition
59.2 of \cite{Con00}). Thus this will imply the desired result that
$\mathcal{S}$ is weakly closed.

\subsection{Approximate factorization}
\begin{lem}
\label{lem:vecs1}Given unit vectors $x,z_{1},...,z_{q}$ in $H$
and $\epsilon>0$, there are vectors $\xi,\zeta_{1},...,\zeta_{q}$
in $F_{n}^{2}$ such that
\begin{enumerate}
\item $\|\xi\|<\sqrt{q}(1+\epsilon)^{1/2}$,
\item $\|\zeta_{i}\|<(1+\epsilon)^{1/2}$ for $1\leq i\leq q$,
\item $[x\otimes z_{i}]_{\mathcal{T}}=\phi([\xi\otimes\zeta_{i}]_{\mathcal{M}_{n}})$
for $1\leq i\leq q$.
\end{enumerate}
\end{lem}
\begin{proof}
Since $\mathcal{M}_{n}$ has property $\mathbb{A}_{1}(1)$, there
are vectors $\upsilon_{1}',...,\upsilon_{q}',\zeta_{1}',...,\zeta_{q}'$
in $F_{n}^{2}$ such that $\|\upsilon_{i}'\|<(1+\epsilon)^{1/2}$,
$\|\zeta_{i}'\|<(1+\epsilon)^{1/2}$ and $[x\otimes z_{i}]_{\mathcal{T}}=\phi([\upsilon_{i}'\otimes\zeta_{i}']_{\mathcal{M}_{n}})$
for $1\leq i\leq q$.

Let $V_{i}=R_{12^{k}}$ for $1\leq i\leq q,$ so that $V_{1},...,V_{q}$
are isometries in $\mathcal{R}_{n}$ with pairwise orthogonal ranges.
Set $\xi=\sum_{i=1}^{q}V_{i}\upsilon_{i}'$ and $\zeta_{i}=V_{i}\zeta_{i}'$
for $1\leq i\leq q$. Then $\|\xi\|<\sqrt{q}(1+\epsilon)^{1/2}$,
$\|\zeta_{i}\|<(1+\epsilon)^{1/2}$ and for $T$ in $\mathcal{T}$,
\begin{eqnarray*}
(\phi([\xi\otimes\zeta_{i}]_{\mathcal{M}_{n}}),T) & = & (\Phi^{-1}(T)\xi,\zeta_{i})\\
 & = & (\Phi^{-1}(T)\sum_{j=1}^{q}V_{j}\upsilon_{j}',V_{i}\zeta_{i}')\\
 & = & (\Phi^{-1}(T)\upsilon_{i}',\zeta_{i}')\\
 & = & (\phi([\upsilon_{i}'\otimes\zeta_{i}']_{\mathcal{M}_{n}}),T)\\
 & = & ([x\otimes z_{i}]_{\mathcal{T}},T).
\end{eqnarray*}
Hence $[x\otimes z_{i}]_{\mathcal{T}}=\phi([\xi\otimes\zeta_{i}]_{\mathcal{M}_{n}})$.
\end{proof}

\begin{lem}
\textup{\label{lem:words-isom}}Let $\eta$ be a unit vector contained
in the algebraic span of $\{\xi_{w}:w\in\mathbb{F}_{n}^{*}\}$. Then
there are words $u$ and $v$ in $\mathbb{F}_{n}^{*}$ such that
\[
L_{u}R_{v}\eta=L\xi_{\varnothing}=R\xi_{\varnothing},
\]
where $L$ is an isometry in $\mathcal{L}_{n}$, and $R$ is an isometry
in $\mathcal{R}_{n}$ with range orthogonal to the range of $R_{1}$.\end{lem}
\begin{proof}
Expand $\eta$ as $\eta=\sum_{|w|\leq m}a_{w}\xi_{w}$ for some $m\geq0$.
Let $u=12^{m}$ and let $v=1^{m}2$. Then $L_{u}R_{v}\eta=\sum_{|w|\leq m}a_{w}\xi_{uwv}$.
Set $L=\sum_{|w|\leq m}a_{w}L_{uwv}$ and $R=\sum_{|w|\leq m}a_{w}R_{uwv}$.
Then $L_{u}R_{v}\eta=L\xi_{\varnothing}=R\xi_{\varnothing}$, and
it's clear that the range of $R$ is orthogonal to the range of $R_{1}$.

It remains to show that $L$ and $R$ are isometries. For $w$ and
$w'$ in $\mathbb{F}_{n}^{+}$ with $|w|\leq m$ and $|w'|\leq m$,
\[
L_{v}^{*}L_{w}^{*}L_{w'}L_{v}=\begin{cases}
I & \mbox{if }w=w',\\
0 & \mbox{otherwise}.
\end{cases}
\]
This gives 
\begin{eqnarray*}
L^{*}L & = & \sum_{|w|\leq m}\sum_{|w'|\leq m}\overline{a_{w}}a_{w'}L_{uwv}^{*}L_{uw'v}\\
 & = & \sum_{|w|\leq m}\sum_{|w'|\leq m}\overline{a_{w}}a_{w'}L_{v}^{*}L_{w}^{*}L_{w'}L_{v}\\
 & = & \sum_{|w|\leq m}|a_{w}|^{2}I\\
 & = & I,
\end{eqnarray*}
where the last equality follows from the fact that $\eta$ is a unit
vector. Thus $L$ is an isometry, and it follows from a similar computation
that $R$ is an isometry.
\end{proof}

\begin{lem}
\label{lem:vecs2}Given unit vectors $z_{1},...,z_{q}$ in $H$ and
$\epsilon>0$, there exists a unit vector $x$ in $H$ and vectors
$\xi,\zeta_{1},...,\zeta_{q}$ in $F_{n}^{2}$ such that
\begin{enumerate}
\item $\|\xi\|<\sqrt{q}(1+\epsilon)^{1/2}$,
\item $\|\zeta_{i}\|<(1+\epsilon)^{1/2}$ for $1\leq i\leq q$,
\item $\xi=\|\xi\|L\xi_{\varnothing}=\|\xi\|R\xi_{\varnothing}$, where
$L$ is an isometry in $\mathcal{L}_{n}$, and $R$ is an isometry
in $\mathcal{R}_{n}$ with range orthogonal to the range of $R_{1}$, 
\item $\|[x\otimes z_{i}]_{\mathcal{T}}-\phi([\xi\otimes\zeta_{i}]_{\mathcal{M}_{n}})\|<\epsilon$
for $1\leq i\leq q$.
\end{enumerate}
\end{lem}

\begin{proof}
Let $x'$ be any unit vector in $H$. By Lemma \ref{lem:vecs1}, there
are vectors $\xi',\zeta_{1}',...,\zeta_{q}'$ in $F_{n}^{2}$ such
that
\begin{enumerate}
\item $\|\xi'\|<\sqrt{q}(1+\epsilon)^{1/2}$,
\item $\|\zeta_{i}'\|<(1+\epsilon)^{1/2}$ for $1\leq i\leq q$,
\item $[x'\otimes z_{i}]_{\mathcal{T}}=\phi([\xi'\otimes\zeta_{i}']_{\mathcal{M}_{n}})$
for $1\leq i\leq q$.
\end{enumerate}
Let $\eta$ be a vector contained in the algebraic span of $\{\xi_{w}:w\in\mathbb{F}_{n}^{*}\}$
such that $\|\eta\|<\sqrt{q}(1+\epsilon)^{1/2}$ and $\|\xi'-\eta\|<\epsilon/(1+\epsilon)^{1/2}$.
Then 
\begin{eqnarray*}
\|[x'\otimes z_{i}]_{\mathcal{T}}-\phi([\eta\otimes\zeta_{i}']_{\mathcal{M}_{n}})\| & \leq & \|[x'\otimes z_{i}]_{\mathcal{T}}-\phi([\xi'\otimes\zeta_{i}']_{\mathcal{M}_{n}})\|\\
 &  & +\|[(\xi'-\eta)\otimes\zeta_{i}']_{\mathcal{M}_{n}}\|\\
 & \leq & \|\xi'-\eta\|\|\zeta_{i}'\|\\
 & < & \epsilon
\end{eqnarray*}
 for $1\leq i\leq q$.

By Lemma \ref{lem:words-isom}, there are words $u$ and $v$ in $\mathbb{F}_{n}^{+}$
such that 
\[
L_{u}R_{v}\eta=\|\eta\|L\xi_{\varnothing}=\|\eta\|R\xi_{\emptyset},
\]
 where $L$ is an isometry in $\mathcal{L}_{n}$, and $R$ is an isometry
in $\mathcal{R}_{n}$ with range orthogonal to the range of $R_{1}$.
Set $x=S_{u}x'$, $\xi=L_{u}R_{v}\eta$ and $\zeta_{i}=R_{v}\zeta_{i}'$
for $1\leq i\leq q$. Then for $T$ in $\mathcal{T}$,
\begin{eqnarray*}
|([x\otimes z_{i}]_{\mathcal{T}}-\phi([\xi\otimes\zeta_{i}]_{\mathcal{M}_{n}}),T)| & = & |([S_{u}x'\otimes z_{i}]_{\mathcal{T}}-\phi([\eta\otimes\zeta_{i}']_{\mathcal{M}_{n}}),T)|\\
 & = & |([x'\otimes z_{i}]_{\mathcal{T}}-\phi([\eta\otimes\zeta_{i}']_{\mathcal{M}_{n}}),TS_{u})|\\
 & \leq & \|[x'\otimes z_{i}]_{\mathcal{T}}-\phi([\eta\otimes\zeta_{i}']_{\mathcal{M}_{n}})\|\|TS_{u}\|\\
 & < & \epsilon\|T\|.
\end{eqnarray*}
Hence $\|[x\otimes z_{i}]_{\mathcal{T}}-\phi([\xi\otimes\zeta_{i}]_{\mathcal{M}_{n}})\|<\epsilon$. 
\end{proof}

The following result is implied by Lemma 1.2 in \cite{Kri01}.
\begin{lem}
\label{lem:strongly--2}Given a proper isometry $R$ in $\mathcal{R}_{n}$,
vectors $\zeta_{1},...,\zeta_{q}$ in $F_{n}^{2}$ and $\epsilon>0$,
there exists $k\geq1$ such that $\|(R^{*})^{k}\zeta_{i}\|<\epsilon$
for $1\leq i\leq q$.
\end{lem}

\begin{lem}
\label{lem:weakly}Given a proper isometry $S$ in $\mathcal{S}$,
vectors $u$ and $v$ in $H$ and $\epsilon>0$, there exists $k\geq1$
such that $\|[u\otimes(S^{*})^{k}v]_{\mathcal{S}}\|<\epsilon$.\end{lem}
\begin{proof}
Since $\mathcal{L}_{n}$ has property $\mathbb{A}_{1},$ there are
vectors $\mu$ and $\nu$ in $F_{n}^{2}$ such that $[u\otimes v]_{\mathcal{S}}=\phi([\mu\otimes\nu]_{\mathcal{L}_{n}})$.
Thus for $A$ in $\mathcal{S}$,
\begin{eqnarray*}
|([u\otimes(S^{*})^{k}v]_{\mathcal{S}},A)| & = & |([\mu\otimes(\Phi^{-1}(S)^{*})^{k}\nu]_{\mathcal{L}_{n}},\Phi^{-1}(A))|\\
 & = & |(\Phi^{-1}(A)\mu,(\Phi^{-1}(S)^{*})^{k}\nu)|\\
 & \leq & \|A\|\|\mu\|\|(\Phi^{-1}(S)^{*})^{k}\nu\|,
\end{eqnarray*}
which gives $\|[u\otimes(S^{*})^{k}v]_{\mathcal{S}}\|\leq\|\mu\|\|(\Phi^{-1}(S)^{*})^{k}\nu\|$.
Since $\Phi^{-1}(S)$ is a a proper isometry in $\mathcal{L}_{n}$,
and since $\mathcal{L}_{n}$ and $\mathcal{R}_{n}$ are unitarily
equivalent, the result now follows by Lemma \ref{lem:strongly--2}.
\end{proof}

\begin{lem}
\label{lem:vecs3}Given unit vectors $z_{1},...,z_{q}$ in $H$ and
$\epsilon>0$, there exists a unit vector $x$ in $H$ and vectors
$\xi,\zeta_{1},...,\zeta_{q}$ in $F_{n}^{2}$ such that
\begin{enumerate}
\item $\|\xi\|<\sqrt{q}(1+\epsilon)^{1/2}$,
\item $\|\zeta_{i}\|<(1+\epsilon)^{1/2}$ for $1\leq i\leq q$,
\item $\xi=\|\xi\|L\xi_{\varnothing}=\|\xi\|R\xi_{\varnothing}$, where
$L$ is an isometry in $\mathcal{L}_{n}$, and $R$ is an isometry
in $\mathcal{R}_{n}$ with range orthogonal to the range of $R_{1}$,
\item $\|R^{*}\zeta_{i}\|<\epsilon$ for $1\leq i\leq q$,
\item $|(\Phi(L)^{k}x,x)|<\epsilon$ for $k\geq1$,
\item $\|[x\otimes z_{i}]_{\mathcal{T}}-\phi([\xi\otimes\zeta_{i}]_{\mathcal{M}_{n}})\|<\epsilon$
for $1\leq i\leq q$.
\end{enumerate}
\end{lem}

\begin{proof}
By Lemma \ref{lem:vecs2}, there exists a unit vector $x'$ in $H$
and vectors $\xi',\zeta_{1},...,\zeta_{q}$ in $F_{n}^{2}$ such that
\begin{enumerate}
\item $\|\xi'\|<\sqrt{q}(1+\epsilon)^{1/2}$,
\item $\|\zeta_{i}\|<(1+\epsilon)^{1/2}$ for $1\leq i\leq q$,
\item $\xi'=\|\xi'\|L'\xi_{\varnothing}=\|\xi'\|R'\xi_{\varnothing}$, where
$L'$ is an isometry in $\mathcal{L}_{n}$ and $R'$ is an isometry
in $\mathcal{R}_{n}$with the range of $R'$ orthogonal to the range
of $R_{1}$, 
\item $\|[x'\otimes z_{i}]_{\mathcal{T}}-\phi([\xi'\otimes\zeta_{i}]_{\mathcal{M}_{n}})\|<\epsilon$
for $1\leq i\leq q$.
\end{enumerate}
By Lemma \ref{lem:strongly--2} and Lemma \ref{lem:weakly}, there
exists $m\geq1$ such that $\|(R_{1}^{*})^{m}(R')^{*}\zeta_{i}\|<\epsilon$
for $1\leq i\leq q$ and $\|[x'\otimes(S_{1}^{*})^{m}\Phi(L')^{*}x']_{\mathcal{S}}\|<\epsilon$.
Set $\xi=L_{1}^{m}\xi'$, $L=L_{1}^{m}L'$ and $R=R'R_{1}^{m}$. Then
$\xi=\|\xi\|L\xi_{\varnothing}=\|\xi\|R\xi_{\varnothing}$, $L$ is
an isometry in $\mathcal{L}_{n}$, and $R$ is an isometry in $\mathcal{R}_{n}$
with range orthogonal to the range of $R_{1}$. For $1\leq i\leq q$,
this gives $\|R^{*}\zeta_{i}\|=\|(R_{1}^{*})^{m}(R')^{*}\zeta_{i}\|<\epsilon$.

Let $x=S_{1}^{m}x'$. Then for $k\geq1$, we compute 
\begin{eqnarray*}
|(\Phi(L)^{k}x,x)| & = & |(\Phi(L_{1}^{m}L')^{k}S_{1}^{m}x',S_{1}^{m}x')|\\
 & = & |(S_{1}^{m}\Phi(L'L_{1}^{m})^{k}x',S_{1}^{m}x')|\\
 & = & |(\Phi(L'L_{1}^{m})^{k}x',x')|\\
 & = & |(\Phi(L'L_{1}^{m})^{k-1}x',(S_{1}^{*})^{m}\Phi(L')^{*}x')|\\
 & = & |([x'\otimes(S_{1}^{*})^{m}\Phi(L')^{*}x']_{\mathcal{S}},\Phi(L'L_{1}^{m})^{k-1})|\\
 & \leq & \|[x'\otimes(S_{1}^{*})^{m}\Phi(L')^{*}x']_{\mathcal{S}}\|\|(L'L_{1}^{m})^{k-1}\|\\
 & < & \epsilon.
\end{eqnarray*}
Finally, for $T$ in $\mathcal{T}$ we have 
\begin{eqnarray*}
|([x\otimes z_{i}]_{\mathcal{T}}-\phi([\xi\otimes\zeta_{i}]_{\mathcal{M}_{n}}),T)| & = & |([x'\otimes z_{i}]_{\mathcal{T}}-\phi([\xi'\otimes\zeta_{i}]_{\mathcal{M}_{n}}),TS_{1}^{m})|\\
 & \leq & \|[x'\otimes z_{i}]_{\mathcal{T}}-\phi([\xi'\otimes\zeta_{i}]_{\mathcal{M}_{n}}\|\|TS_{1}^{m}\|\\
 & < & \epsilon\|T\|.
\end{eqnarray*}
Thus $\|[x\otimes z_{i}]_{\mathcal{T}}-\phi([\xi\otimes\zeta_{i}]_{\mathcal{M}_{n}})\|<\epsilon$.
\end{proof}

\subsection{Approximately orthogonal vectors}

The following lemma is extracted from the proof of Theorem 4.3 in
\cite{Ber98}.
\begin{lem}
\label{lem:bercovici--2} Given two isometries $R$ and $R'$ in $\mathcal{R}_{n}$
with orthogonal ranges and vectors $\xi$ and $\mu$ in $F_{n}^{2}$
with $\mu$ in the kernel of $R^{*}$, define
\[
\mu_{k}=\frac{1}{\sqrt{k}}\sum_{j=1}^{k}R^{j}R'\mu.
\]
Then 
\[
\|[\xi\otimes\mu_{k}]_{\mathcal{M}_{n}}\|\leq\frac{1}{\sqrt{k}}\|\mu\|\|D_{k}-1\|_{1},
\]
where $D_{k}$ denotes the $k$-th Dirichlet kernel and $\|\cdot\|_{1}$
denotes the $L^{1}$ norm. 
\end{lem}

\begin{lem}
\label{lem:approx-orthog}Given unit vectors $z_{1},...,z_{q}$ in
$H$ and $\epsilon>0$, there exists a unit vector $x$ in $H$ such
that $\|[x\otimes z_{i}]_{\mathcal{T}}\|<\epsilon$ for $1\leq i\leq q$.\end{lem}
\begin{proof}
We may suppose that $\epsilon<1$. Using the fact that $\lim k^{-1/2}\|D_{k}\|_{1}=0$,
where $D_{k}$ denotes the $k$-th Dirichlet kernel and $\|\cdot\|_{1}$
denotes the $L^{1}$ norm, choose $k\geq1$ such that $2(q/k)^{-1/2}\|D_{k}-1\|_{1}<\epsilon/(3(1+\epsilon)$).
Next choose $\epsilon'>0$ such that 
\[
\epsilon'<\min\left\{ 1,\frac{\epsilon(1-\epsilon)}{3\sqrt{k}},\frac{\epsilon(1-\epsilon)}{6\sqrt{q}},\frac{k\epsilon}{k^{2}-k}\right\} .
\]
By Lemma \ref{lem:vecs3}, there exists a unit vector $x'$ in $H$
and vectors $\xi',\zeta_{1},...,\zeta_{q}$ in $F_{n}^{2}$ such that
\begin{enumerate}
\item $\|\xi'\|<\sqrt{q}(1+\epsilon')^{1/2}$,
\item $\|\zeta_{i}\|<(1+\epsilon')^{1/2}$ for $1\leq i\leq q$,
\item \label{enu:approx-orthog-3}$\xi'=\|\xi'\|L\xi_{\varnothing}=\|\xi'\|R\xi_{\varnothing}$,
where $L$ is an isometry in $\mathcal{L}_{n}$, and $R$ is an isometry
in $\mathcal{R}_{n}$ with range orthogonal to the range of $R_{1}$,
\item \label{enu:approx-orthog-4}$\|R^{*}\zeta_{i}\|<\epsilon'$ for $1\leq i\leq q$,
\item $|(\Phi(L)^{k}x',x')|<\epsilon'$ for $k\geq1$,
\item $\|[x'\otimes z_{i}]_{\mathcal{T}}-\phi([\xi'\otimes\zeta_{i}]_{\mathcal{M}_{n}})\|<\epsilon'$
for $1\leq i\leq q$.
\end{enumerate}
By \eqref{enu:approx-orthog-4} we can write $\zeta_{i}=\mu_{i}+\nu_{i}$,
where $\mu_{i}$ is in the kernel of $R^{*}$ and $\|\nu_{i}\|<\epsilon'$.

Let $\xi=k^{-1/2}\sum_{j=0}^{k-1}L_{1}L^{j}\xi'$. Then by \eqref{enu:approx-orthog-3}
we can write $\xi$ as
\begin{eqnarray*}
\xi & = & \frac{1}{\sqrt{k}}\sum_{j=0}^{k-1}L_{1}L^{j}\xi'\\
 & = & \frac{\|\xi'\|}{\sqrt{k}}\sum_{j=0}^{k-1}L_{1}L^{j+1}\xi_{\varnothing}\\
 & = & \frac{\|\xi'\|}{\sqrt{k}}\sum_{j=1}^{k}L_{1}L^{j}\xi_{\varnothing},
\end{eqnarray*}
which implies $\|\xi\|=\|\xi'\|$. Applying \eqref{enu:approx-orthog-3}
again, we can also write $\xi$ as
\begin{eqnarray*}
\xi & = & \frac{1}{\sqrt{k}}\sum_{j=0}^{k-1}L_{1}L^{j}\xi'\\
 & = & \frac{\|\xi'\|}{\sqrt{k}}\sum_{j=0}^{k-1}L_{1}L^{j}R\xi_{\varnothing}\\
 & = & \frac{\|\xi'\|}{\sqrt{k}}\sum_{j=0}^{k-1}R^{j+1}R_{1}\xi_{\varnothing}\\
 & = & \frac{\|\xi'\|}{\sqrt{k}}\sum_{j=1}^{k}R^{j}R_{1}\xi_{\varnothing}.
\end{eqnarray*}
By Lemma \ref{lem:bercovici--2} and the choice of $k$, this gives
\begin{eqnarray*}
\|[\xi\otimes\mu_{i}]_{\mathcal{M}_{n}}\| & \leq & \frac{1}{\sqrt{k}}\|\xi\|\|\mu_{i}\|\|D_{k}-1\|_{1}\\
 & \leq & \sqrt{\frac{q}{k}}(1+\epsilon')^{1/2}\|\mu_{i}\|\|D_{k}-1\|_{1}\\
 & < & \epsilon(1-\epsilon)/3.
\end{eqnarray*}

Let $y=Sx'$, where $S=k^{-1/2}\sum_{j=0}^{k-1}\Phi(L_{1}L^{j})$.
Then $\|S\|\leq\sqrt{k}$, so for $T$ in $\mathcal{T}$,

\begin{eqnarray*}
|([y\otimes z_{i}]_{\mathcal{T}}-\phi([\xi\otimes\zeta_{i}]_{\mathcal{M}_{n}}),T)| & = & |([x'\otimes z_{i}]_{\mathcal{T}}-\phi([\xi'\otimes\zeta_{i}]_{\mathcal{M}_{n}}),TS)|\\
 & \leq & \|[x'\otimes z_{i}]_{\mathcal{T}}-\phi([\xi'\otimes\zeta_{i}]_{\mathcal{M}_{n}})\|\|TS\|\\
 & < & \epsilon'\sqrt{k}\|T\|\\
 & < & (\epsilon(1-\epsilon)/3)\|T\|,
\end{eqnarray*}
which gives $\|[y\otimes z_{i}]_{\mathcal{T}}-\phi([\xi\otimes\zeta_{i}]_{\mathcal{M}_{n}})\|<\epsilon(1-\epsilon)/3$.
Since 
\[
\|[\xi\otimes\nu_{i}]_{\mathcal{M}_{n}}\|\leq\|\xi\|\|\nu_{i}\|<\sqrt{q}(1+\epsilon')^{1/2}\epsilon'<\epsilon(1-\epsilon)/3,
\]
this gives
\begin{eqnarray*}
\|[y\otimes z_{i}]_{\mathcal{T}}\| & \leq & \|[y\otimes z_{i}]_{\mathcal{T}}-\phi([\xi\otimes\zeta_{i}]_{\mathcal{M}_{n}})\|+\|[\xi\otimes\mu_{i}]_{\mathcal{M}_{n}}\|+\|[\xi\otimes\nu_{i}]_{\mathcal{M}_{n}}\|\\
 & < & \epsilon(1-\epsilon).
\end{eqnarray*}
Finally, we compute
\begin{eqnarray*}
\|y\|^{2} & = & \|Sx'\|^{2}\\
 & = & \left\Vert \frac{1}{\sqrt{k}}\sum_{j=0}^{k-1}\Phi(L_{1}L^{j})x'\right\Vert ^{2}\\
 & = & \|x'\|^{2}+\frac{1}{k}\sum_{0\leq i<j\leq k-1}(x',\Phi(L)^{j-i}x')+\frac{1}{k}\sum_{0\leq j<i\leq k-1}(\Phi(L)^{i-j}x',x')\\
 & \geq & 1-\frac{1}{k}\sum_{0\leq i<j\leq k-1}|(x',\Phi(L)^{j-i}x')|-\frac{1}{k}\sum_{0\leq j<i\leq k-1}|(\Phi(L)^{i-j}x',x')|\\
 & \geq & 1-\frac{k^{2}-k}{k}\epsilon'.\\
 & > & 1-\epsilon.
\end{eqnarray*}
Hence taking $x=(1-\epsilon)^{-1}y$, $\|x\|\geq1$ and $\|[x\otimes z_{i}]_{\mathcal{T}}\|<\epsilon$
for $1\leq i\leq q$. 
\end{proof}

\begin{lem}
\label{lem:approx-orthog-inter}Given unit vectors $z_{1},...,z_{q}$
in $H$ and $\epsilon>0$, there exists an intertwining operator $X:F_{n}^{2}\to H$
such that $\|X\xi_{\varnothing}\|=1$ and $\|[X\xi_{\varnothing}\otimes z_{i}]_{\mathcal{T}}\|<\epsilon$
for $1\leq i\leq q$.\end{lem}
\begin{proof}
By Lemma \ref{lem:approx-orthog}, there exists a unit vector $x$
in $H$ such that $\|[x\otimes z_{i}]_{\mathcal{T}}\|<\epsilon$ for
$1\leq i\leq q$. By Theorem 2.7 of \cite{DLP05}, $x$ is in the
range of an intertwining operator $X':F_{n}^{2}\to H$. Hence there
is a vector $\xi$ in $F_{n}^{2}$ such that $X'\xi=x$. The result
now follows from the fact that the set of vectors $\{R\xi_{\varnothing}:R\in\mathcal{R}_{n}\}$
is dense in $F_{n}^{2}$, and the fact that for $R$ in $\mathcal{R}_{n}$,
the operator $X'R$ is intertwining.
\end{proof}

\begin{lem}
\label{lem:approx-functional}Let $X:F_{n}^{2}\to H$ be an intertwining
operator with $\|X\xi_{\varnothing}\|=1$. Then given $\epsilon>0$,
there is a word $v$ in $\mathbb{F}_{n}^{*}$ such that 
\[
\|[XR_{v}\xi_{\varnothing}\otimes XR_{v}\xi_{\varnothing}]_{\mathcal{T}}-\phi([\xi_{\varnothing}\otimes\xi_{\varnothing}]_{\mathcal{M}_{n}})\|<\epsilon.
\]
\end{lem}
\begin{proof}
Since $X^{*}X$ is an L-Toeplitz operator, by Lemma 4.5 of \cite{Ken11},
there is a word $v$ in $\mathbb{F}_{n}^{*}$ such that $\|R_{v}^{*}X^{*}XR_{v}\xi_{\varnothing}-\xi_{\varnothing}\|<\epsilon/2$.
Note that $R_{v}^{*}X^{*}XR_{v}$ is also an L-Toeplitz operator.
Let $\xi=(R_{v}^{*}X^{*}XR_{v}-I)\xi_{\varnothing}$, so that $\|\xi\|<\epsilon/2$.
For $w$ in $\mathbb{F}_{n}^{*}$, since $(L_{w}\xi,\xi_{\varnothing})=0$
we can write 
\begin{eqnarray*}
(S_{w}XR_{v}\xi_{\varnothing},XR_{v}\xi_{\varnothing}) & = & (L_{w}\xi_{\varnothing},R_{v}^{*}X^{*}XR_{v}\xi_{\varnothing})\\
 & = & (L_{w}\xi_{\varnothing},\xi_{\varnothing})+(L_{w}\xi_{\varnothing},\xi)+(L_{w}\xi,\xi_{\varnothing}).
\end{eqnarray*}
Similarly,
\begin{eqnarray*}
(S_{w}^{*}XR_{v}\xi_{\varnothing},XR_{v}\xi_{\varnothing}) & = & (L_{w}^{*}R_{v}^{*}X^{*}XR_{v}\xi_{\varnothing},\xi_{\varnothing})\\
 & = & (L_{w}^{*}\xi_{\varnothing},\xi_{\varnothing})+(L_{w}^{*}\xi_{\varnothing},\xi)+(L_{w}^{*}\xi,\xi_{\varnothing}).
\end{eqnarray*}
This gives
\[
[XR_{v}\xi_{\varnothing}\otimes XR_{v}\xi_{\varnothing}]_{\mathcal{T}}=\phi([\xi_{\varnothing}\otimes\xi_{\varnothing}]_{\mathcal{M}_{n}}+[\xi\otimes\xi_{\varnothing}]_{\mathcal{M}_{n}}+[\xi_{\varnothing}\otimes\xi]_{\mathcal{M}_{n}}),
\]
so we conclude that
\begin{eqnarray*}
\|[XR_{v}\xi_{\varnothing}\otimes XR_{v}\xi_{\varnothing}]_{\mathcal{T}}-\phi([\xi_{\varnothing}\otimes\xi_{\varnothing}]_{\mathcal{M}_{n}})\| & \leq & \|[\xi\otimes\xi_{\varnothing}]_{\mathcal{M}_{n}}+[\xi_{\varnothing}\otimes\xi]_{\mathcal{M}_{n}}\|\\
 & \leq & 2\|\xi\|\|\xi_{\varnothing}\|\\
 & < & \epsilon,
\end{eqnarray*}
as required.
\end{proof}

\begin{lem}
\label{lem:approx-combined}Given unit vectors $z_{1},...,z_{q}$
in $H$ and $\epsilon>0$, there exists an intertwining operator $X:F_{n}^{2}\to H$
such that $\|X\xi_{\varnothing}\|=1$, $\|[X\xi_{\varnothing}\otimes z_{i}]_{\mathcal{T}}\|<\epsilon$
for $1\leq i\leq q$ and 
\[
\|[X\xi_{\varnothing}\otimes X\xi_{\varnothing}]_{\mathcal{T}}-\phi([\xi_{\varnothing}\otimes\xi_{\varnothing}]_{\mathcal{M}_{n}})\|<\epsilon.
\]
\end{lem}
\begin{proof}
By Lemma \ref{lem:approx-orthog-inter}, there exists an intertwining
operator $X':F_{n}^{2}\to H$ such that $\|X'\xi_{\varnothing}\|=1$
and $\|[X'\xi_{\varnothing}\otimes z_{i}]_{\mathcal{T}}\|<\epsilon$
for $1\leq i\leq q$. By Lemma \ref{lem:approx-functional}, there
is a word $v$ in $\mathbb{F}_{n}^{*}$ such that 
\[
\|[X'R_{v}\xi_{\varnothing}\otimes X'R_{v}\xi_{\varnothing}]_{\mathcal{T}}-\phi([\xi_{\varnothing}\otimes\xi_{\varnothing}]_{\mathcal{M}_{n}})\|<\epsilon.
\]
Let $X=X'R_{v}$. Then 
\[
\|X\xi_{\varnothing}\|=\|X'R_{v}\xi_{\varnothing}\|=\|X'L_{v}\xi_{\varnothing}\|=\|S_{v}X'\xi_{\varnothing}\|=\|X'\xi_{\varnothing}\|=1.
\]
For $T$ in $\mathcal{T}$,
\begin{eqnarray*}
|([X\xi_{\varnothing}\otimes z_{i}]_{\mathcal{T}},T)| & = & |([X'R_{v}\xi_{\varnothing}\otimes z_{i}]_{\mathcal{T}},T)|\\
 & = & |([X'L_{v}\xi_{\varnothing}\otimes z_{i}]_{\mathcal{T}},T)|\\
 & = & |([S_{v}X'\xi_{\varnothing}\otimes z_{i}]_{\mathcal{T}},T)|\\
 & = & |([X'\xi_{\varnothing}\otimes z_{i}]_{\mathcal{T}},TS_{v})|\\
 & \leq & \|[X'\xi_{\varnothing}\otimes z_{i}]_{\mathcal{T}}\|\|T\|.
\end{eqnarray*}
Hence $\|[X\xi_{\varnothing}\otimes z_{i}]_{\mathcal{T}}\|<\epsilon$
for $1\leq i\leq q$. 
\end{proof}

\subsection{The strong factorization property}
\begin{thm}
\label{thm:x01}Given a weak-{*} continuous linear functional $\tau$
on $\mathcal{T}$ with $\|\tau\|\leq1$, unit vectors $z_{1},...,z_{q}$
in $H$ and $\epsilon>0$, there are vectors $x$ and $y$ in $H$
such that
\begin{enumerate}
\item $\|x\|\leq1$ and $\|y\|\leq1$,
\item $\|\tau-[x\otimes y]_{\mathcal{T}}\|<\epsilon$,
\item $\|[x\otimes z_{i}]_{\mathcal{T}}\|<\epsilon$ and $\|[z_{i}\otimes y]_{\mathcal{T}}\|<\epsilon$
for $1\leq i\leq q$.
\end{enumerate}
In other words, $\mathcal{T}$ has property $\mathcal{X}_{0,1}$.\end{thm}
\begin{proof}
Choose $\epsilon'>0$ such that $\epsilon'<\epsilon$ and $1-(1+2\epsilon')^{-2}(1-\epsilon')<\epsilon$.
Since $\mathcal{M}_{n}$ has property $\mathbb{A}_{1}(1)$, there
are vectors $\xi$ and $\upsilon$ in $F_{n}^{2}$ with $\|\xi\|\leq1+\epsilon'/2$
and $\|\upsilon\|\leq1+\epsilon'/2$ such that $\tau=\phi([\xi\otimes\upsilon]_{\mathcal{M}_{n}})$.
Since $\xi_{\varnothing}$ is cyclic for $\mathcal{L}_{n}$, there
are $A$ and $B$ in $\mathcal{L}_{n}$ such that $\|A\xi_{\varnothing}-\xi\|<\epsilon'/(4(1+\epsilon'))$
and $\|B\xi_{\varnothing}-\upsilon\|<\epsilon'/(4(1+\epsilon'))$.
Then 
\[
\|A\xi_{\varnothing}\|\leq\|A\xi_{\varnothing}-\xi\|+\|\xi\|<1+\epsilon',
\]
and similarly $\|B\xi_{\varnothing}\|<1+\epsilon'$. This gives
\begin{eqnarray*}
\|[A\xi_{\varnothing}\otimes B\xi_{\varnothing}]_{\mathcal{M}_{n}}-[\xi\otimes\upsilon]_{\mathcal{M}_{n}}\| & \leq & \|[(A\xi_{\varnothing}-\xi)\otimes B\xi_{\varnothing}]\|\\
 &  & +\|[\xi\otimes(B\xi_{\varnothing}-\upsilon)]_{\mathcal{M}_{n}}\|\\
 & \leq & \|A\xi_{\varnothing}-\xi\|\|B\xi_{\varnothing}\|+\|\xi\|\|B\xi_{\varnothing}-\upsilon\|\\
 & < & \epsilon'/2.
\end{eqnarray*}

By Lemma \ref{lem:approx-combined}, there is an intertwining operator
$X:F_{n}^{2}\to H$ such that $\|X\xi_{\varnothing}\|=1$, $\|[X\xi_{\varnothing}\otimes z_{i}]_{\mathcal{T}}\|<\epsilon'/(\|A\|+\|B\|)$
for $1\leq i\leq q$ and $\|[X\xi_{\varnothing}\otimes X\xi_{\varnothing}]_{\mathcal{T}}-\phi([\xi_{\varnothing}\otimes\xi_{\varnothing}]_{\mathcal{M}_{n}})\|<\epsilon'/(2(\|A\|+\|B\|)^{2}).$
Note that since $\mathcal{T}$ is self-adjoint, we also have $\|[z_{i}\otimes X\xi_{\varnothing}]_{\mathcal{T}}\|<\epsilon'/(\|A\|+\|B\|)$
for $1\leq i\leq q$.

Define vectors $x'$ and $y'$ in $H$ by $x'=\Phi(A)X\xi_{\varnothing}$
and $y'=\Phi(B)X\xi_{\varnothing}$. Then 
\begin{eqnarray*}
\|x'\|^{2} & = & \|\Phi(A)X\xi_{\varnothing}\|^{2}\\
 & = & \|\Phi(A)X\xi_{\varnothing}\|^{2}-\|A\xi_{\varnothing}\|^{2}+\|A\xi_{\varnothing}\|^{2}\\
 & = & |([X\xi_{\varnothing}\otimes X\xi_{\varnothing}]_{\mathcal{T}}-\phi([\xi_{\varnothing}\otimes\xi_{\varnothing}]_{\mathcal{M}_{n}}),\Phi(A^{*}A))|+\|A\xi_{\varnothing}\|^{2}\\
 & \leq & \|[X\xi_{\varnothing}\otimes X\xi_{\varnothing}]_{\mathcal{T}}-\phi([\xi_{\varnothing}\otimes\xi_{\varnothing}]_{\mathcal{M}_{n}})\|\|A\|^{2}+\|A\xi_{\varnothing}\|^{2}\\
 & < & 1+2\epsilon',
\end{eqnarray*}
and similarly, $\|y'\|^{2}<1+2\epsilon'$. For $T$ in $\mathcal{T}$,

\begin{align*}
|([x'\otimes y']_{\mathcal{T}}-\phi([A\xi_{\varnothing} & \otimes B\xi_{\varnothing}]_{\mathcal{M}_{n}}),T)|\\
 & =|([\Phi(A)X\xi_{\varnothing}\otimes\Phi(B)X\xi_{\varnothing}]_{\mathcal{T}}-\phi([A\xi_{\varnothing}\otimes B\xi_{\varnothing}]_{\mathcal{M}_{n}}),T)|\\
 & =|([X\xi_{\varnothing}\otimes X\xi_{\varnothing}]_{\mathcal{T}}-\phi([\xi_{\varnothing}\otimes\xi_{\varnothing}]_{\mathcal{M}_{n}}),\Phi(A)^{*}T\Phi(B))|\\
 & \leq\|[X\xi_{\varnothing}\otimes X\xi_{\varnothing}]_{\mathcal{T}}-\phi([\xi_{\varnothing}\otimes\xi_{\varnothing}]_{\mathcal{M}_{n}})\|\|A\|\|B\|\|T\|\\
 & <\frac{\epsilon'}{2}\|T\|,
\end{align*}
which implies $\|[x'\otimes y']_{\mathcal{T}}-\phi([A\xi_{\varnothing}\otimes B\xi_{\varnothing}]_{\mathcal{M}_{n}})\|<\epsilon'/2$.
Thus
\begin{eqnarray*}
\|[x'\otimes y']_{\mathcal{T}}-\tau\| & = & \|[x'\otimes y']_{\mathcal{T}}-\phi([\xi\otimes\upsilon]_{\mathcal{M}_{n}})\|\\
 & \leq & \|[x'\otimes y']_{\mathcal{T}}-\phi([A\xi_{\varnothing}\otimes B\xi_{\varnothing}]_{\mathcal{M}_{n}})\|\\
 &  & +\|[A\xi_{\varnothing}\otimes B\xi_{\varnothing}]_{\mathcal{M}_{n}}-[\xi\otimes\upsilon]_{\mathcal{M}_{n}}\|\\
 & < & \epsilon'.
\end{eqnarray*}
For $1\leq i\leq q$, 
\[
\|[x'\otimes z_{i}]_{\mathcal{T}}\|=\|[AX\xi_{\varnothing}\otimes z_{i}]_{\mathcal{T}}\|\leq\|A\|\|[X\xi_{\varnothing}\otimes z_{i}]_{\mathcal{T}}\|<\epsilon',
\]
and similarly, $\|[z_{i}\otimes y']_{\mathcal{T}}\|<\epsilon'$.

Now take $x=(1+2\epsilon')^{-1}x'$ and $y=(1+2\epsilon')^{-1}y'$.
Then by choice of $\epsilon'$ we get $\|x\|\leq1$ and $\|y\|\leq1$.
Similarly, $\|[x\otimes z_{i}]_{\mathcal{T}}\|<\epsilon$ and $\|[z_{i}\otimes y]_{\mathcal{T}}\|<\epsilon$
for $1\leq i\leq q$. Finally, we have 
\begin{eqnarray*}
\|[x\otimes y]_{\mathcal{T}}-\tau\| & \leq & (1+2\epsilon')^{-2}\|[x'\otimes y']_{\mathcal{T}}-\tau\|+(1-(1+2\epsilon')^{-2})\|\tau\|\\
 & < & 1-(1+2\epsilon')^{-2}(1-\epsilon')\\
 & < & \epsilon,
\end{eqnarray*}
as required.
\end{proof}

\subsection{Absolute continuity and analyticity}
\begin{thm}
\label{thm:abs-cont-is-anal}For $n\geq2$, every absolutely continuous
isometric $n$-tuple is analytic.\end{thm}
\begin{proof}
For $n\geq2$, let $S=(S_{1},\ldots,S_{n})$ be an absolutely continuous
isometric $n$-tuple, and let $\mathcal{S}$ denote the weak-{*} closed
unital algebra generated by $S_{1},\ldots,S_{n}$. By Corollary 1.2
of \cite{DY08}, $\mathcal{S}$ is isomorphic to the noncommutative
analytic Toeplitz algebra $\mathcal{L}_{n}$. By Theorem \ref{thm:x01},
$\mathcal{S}$ has property $\mathcal{X}_{0,1}$, and hence has property
$\mathbb{A}_{1}(1)$. Therefore, by the discussion in Section \ref{sub:dual-alg-theory-prelim},
$\mathcal{S}$ is weakly closed, and hence $\mathcal{S}$ is actually
the free semigroup algebra (i.e. the weakly closed algebra) generated
by $S_{1},\ldots,S_{n}$. Since $\mathcal{S}$ is isomorphic to $\mathcal{L}_{n}$,
this implies that $S$ is analytic.
\end{proof}

The next result follow from Theorem 4.12 of \cite{Ken11}.
\begin{cor}
\label{cor:wand-vecs-span}For $n\geq2$, let $S=(S_{1},\ldots,S_{n})$
be an absolutely continuous isometric $n$-tuple acting on a Hilbert
space $H$. Then the wandering vectors for $S$ span $H$.
\end{cor}

It was shown in Corollary 5.8 of \cite{Ken11} that every analytic
isometric tuple is hyperreflexive with hyperreflexivity constant at
most $3$, but the next result can also be proved directly using Theorem
\ref{thm:x01} of the present paper and Theorem 3.1 of \cite{Ber98}.
\begin{cor}
Absolutely continuous row isometries are hyperreflexive with hyperreflexivity
constant at most $3$.
\end{cor}

\section{Singular isometric tuples\label{sec:sing-isom-tuple}}

In Theorem \ref{thm:abs-cont-is-anal}, we showed that for $n\geq2$,
an isometric $n$-tuple is absolutely continuous if and only if it
is analytic. With this operator-algebraic characterization of an absolutely
continuous isometric tuple, we are now able to give an operator-algebraic
characterization of a singular isometric tuple.

\begin{thm}
\label{thm:sing-self-adjoint}For $n\geq2$, an isometric $n$-tuple
is singular if and only if the free semigroup algebra it generates
is a von Neumann algebra.
\end{thm}

\begin{proof}
Let $V=(V_{1},\ldots,V_{n})$ be an isometric $n$-tuple, and let
$\mathcal{V}$ denote the free semigroup algebra (i.e. the weakly
closed algebra) generated by $V$. If $\mathcal{V}$ is a von Neumann
algebra, then $V$ has no absolutely continuous part since, by Theorem
\ref{thm:abs-cont-is-anal}, an absolutely continuous isometric tuple
is analytic, and the noncommutative analytic Toeplitz algebra $\mathcal{L}_{n}$
is not self-adjoint by Corollary 1.5 of \cite{DP99}.

Conversely, if $V$ is singular then it has no analytic restriction
to an invariant subspace since, by Theorem \ref{thm:abs-cont-is-anal},
an absolutely continuous isometric tuple is analytic. Thus by Theorem
\ref{thm:struct-thm-fsa}, $\mathcal{V}$ is a von Neumann algebra. 
\end{proof}

Example \ref{exa:reductive-abs-cont-uni} showed that it is possible
for an absolutely continuous unitary to generate a von Neumann algebra.
Theorem \ref{thm:sing-self-adjoint} implies that there is no higher-dimensional
analogue of this phenomenon.

Recall that a family of operators is said to be reductive if every
subspace invariant for the family is also coinvariant.
\begin{cor}
For $n\geq2$, every reductive unitary $n$-tuple is singular.\end{cor}
\begin{proof}
Let $V=(V_{1},\ldots,V_{n})$ be a reductive isometric $n$-tuple,
and let $\mathcal{V}$ denote the free semigroup algebra generated
by $V$. By the dichotomy for free semigroup algebras, Corollary 4.13
of \cite{Ken11}, if $\mathcal{V}$ is not a von Neumann algebra,
then there is a vector $x$ that is wandering for $V$. Let $\mathcal{V}[x]$
denote the cyclic invariant subspace generated by $x$. Then the subspace
$\sum_{i=1}^{n}V_{i}\mathcal{V}[x]$ is invariant for $V$ but not
coinvariant, which would contradict that $V$ is reductive. Thus $\mathcal{V}$
is a von Neumann algebra and $V$ is singular by Theorem \ref{thm:sing-self-adjoint}.
\end{proof}

\begin{example}
By Theorem \ref{thm:sing-self-adjoint}, for $n\geq2$ an isometric
$n$-tuple is singular if and only if the free semigroup algebra it
generates is a von Neumann algebra. The existence of a self-adjoint
free semigroup algebra on two or more generators was conjectured in
\cite{DKP01}, but it took some time for the first example to be constructed.
In \cite{Read05}, Read showed that $B(\ell^{2})$ is generated as
a free semigroup algebra on two generators. In \cite{Dav06}, Davidson
gave an exposition of Read's construction and showed that it could
be generalized to show that $B(\ell^{2})$ is generated as a free
semigroup algebra on $n$ generators for every $n\geq2$. By our characterization
of singularity, this gives an example of a singular isometric $n$-tuple
for every $n\geq2$.
\end{example}

\section{The Lebesgue-von Neumann-Wold decomposition\label{sec:lebesgue-wold-decomp}}

In Theorem \ref{thm:abs-cont-is-anal}, we showed that for $n\geq2$,
an isometric $n$-tuple is absolutely continuous if and only if it
is analytic. In Theorem  \ref{thm:sing-self-adjoint}, we showed that
for $n\geq2$, an isometric $n$-tuple is singular if and only if
the free semigroup algebra (i.e. the weakly closed algebra) it generates
is a von Neumann algebra. With these operator-algebraic characterizations
of absolute continuity and singularity, we will be able to prove the
Lebesgue-von Neumann-Wold decomposition of an isometric tuple.

In the classical case, the Lebesgue decomposition of a measure guarantees
that every unitary splits into absolutely continuous and singular
parts. For $n\geq2$, it turns out that it is possible for a unitary
$n$-tuple to be irreducible and neither absolutely continuous nor
singular.
\begin{defn}
An isometric $n$-tuple $V=(V_{1},\ldots,V_{n})$ is said to be of
\textbf{dilation type} if it has no summand that is absolutely continuous
or singular.

Note that by the Wold decomposition of an isometric tuple, Proposition
\ref{pro:wold-decomp}, an isometric $n$-tuple of dilation type is
necessarily unitary. The next result provides a characterization of
an isometric tuple of dilation type as a minimal dilation, in the
sense of Section \ref{sub:dilation-theory}.\end{defn}
\begin{prop}
\label{pro:struct-dilation-type}Let $V=(V_{1},\ldots,V_{n})$ be
an isometric $n$-tuple of dilation type. Then there is a subspace
$H$ coinvariant under $V$ such that $H$ is cyclic for $V$ and
the compression of $V$ to $H^{\perp}$ is a unilateral $n$-shift.
In other words, $V$ is the minimal isometric dilation of its compression
to $H$.\end{prop}
\begin{proof}
Note that since $V$ has no summand that is absolutely continuous,
by Proposition \ref{pro:wold-decomp}  $V$ is necessarily a unitary
$n$-tuple. Let $\mathcal{V}$ denote the free semigroup algebra generated
by $V$, and let $P$ be the projection from Theorem \ref{thm:struct-thm-fsa}
applied to $\mathcal{V}$. Let $H$ be the range of $P$, so that
$H$ is coinvariant under $V$.

Let $K=(H+\sum_{i=1}^{n}V_{i}H)\ominus H$. Then $K$ is wandering
for the compression of $V$ to $H^{\perp}$. If $K=0$, then by Theorem
\ref{thm:struct-thm-fsa}, $\mathcal{V}$ can be decomposed into the
direct sum of a self-adjoint free semigroup algebra and an analytic
free semigroup algebra. By the characterization of singular isometric
tuples, Corollary \ref{thm:sing-self-adjoint}, this would contradict
that $V$ is of dilation type. Thus $K\ne0$. The fact that $K$ is
cyclic follows from the fact that $H$ is cyclic. 
\end{proof}

\begin{example}
[An irreducible isometric tuple of dilation type]For $n\geq2$, define
isometries $V_{1},\ldots,V_{n}$ on $\ell^{2}(\mathbb{N})$ by
\[
V_{k}e_{l}=e_{n(l-1)+k},
\]
where $\{e_{l}\}_{l=1}^{\infty}$ is the standard orthonormal basis
of $\ell^{2}(\mathbb{N})$. Then the range of each $V_{k}$ is spanned
by the orthonormal set
\[
\{V_{k}e_{l}\}_{l=1}^{\infty}=\{e_{n(l-1)+k}\}_{l=1}^{\infty}=\{e_{l}:l\equiv k\ \mbox{mod}\ n\}.
\]
Therefore, the operators $V_{1},\ldots,V_{n}$ are isometries with
mutually orthogonal ranges, meaning $V=(V_{1},\ldots,V_{n})$ is an
isometric tuple. We will show that $V$ is an irreducible isometric
tuple of dilation type.

Since the vector $e_{1}$ is fixed by $V_{1}$, it is straightforward
to check that the sequence $\{V_{1}^{k}\}_{k=1}^{\infty}$ is weakly
convergent to the rank one projection $e_{1}e_{1}^{*}$. In particular,
$e_{1}e_{1}^{*}$ is contained in the von Neumann algebra $\mathrm{W}^{*}(V)$
generated by $V$. (In fact, this is the projection provided by Theorem
\ref{thm:struct-thm-fsa}). Since the vector $e_{1}$ is cyclic for
$V$, it follows that $\mathrm{W}^{*}(V)=B(\ell^{2}(\mathbb{N})$),
and hence that $V$ is irreducible.

To see that $V$ is of dilation type, it suffices to show that $V$
is neither singular nor absolutely continuous. Since $V$ is irreducible
from above, if $V$ was singular then by Theorem \ref{thm:sing-self-adjoint},
the free semigroup algebra $\mathrm{W}(V)$ generated by $V$ would
be $B(\ell^{2}(\mathbb{N})$. However, the vector $e_{2}$ is wandering
for $V$ and the vector $e_{1}$ is orthogonal to the wandering subspace
spanned by $\{V_{w}e_{2}:w\in\mathbb{F}_{n}^{*}\}$, which implies
that $\mathrm{W}(V)$ is not transitive, and hence that $\mathrm{W}(V)$
is properly contained in $B(\ell^{2}(\mathbb{N})$. Thus $V$ is not
singular. The fact that $V$ is not absolutely continuous follows
from Theorem \ref{thm:abs-cont-is-anal} and the observation made
above that the sequence $\{V_{1}^{k}\}_{k=1}^{\infty}$ is weakly
convergent to the projection $e_{1}e_{1}^{*}$.
\end{example}

\begin{example}
[A family of irreducible isometric tuples of dilation type]It was
shown in Corollary 6.6 of \cite{DKS01} that the minimal isometric
dilation of a contractive $n$-tuple $A=(A_{1},\ldots,A_{n})$ acting
on a finite-dimensional space is an irreducible unitary $n$-tuple
if and only if both $\sum_{i=1}^{n}A_{i}A_{i}^{*}=I$ and $\mathrm{C}^{*}(A)$
has a minimal coinvariant subspace that is cyclic for $\mathrm{C}^{*}(A)$.
These conditions are satisfied, for example, by the contractive tuple
$A=(A_{1},A_{2})$, where 
\[
A_{1}=\left(\begin{array}{cc}
0 & 1\\
0 & 0
\end{array}\right),\qquad A_{2}=\left(\begin{array}{cc}
0 & 0\\
1 & 0
\end{array}\right).
\]
Thus the minimal isometric dilation of $A$ is an example of an irreducible
isometric tuple of dilation type. A similar construction can be carried
out for all $n\geq2$.
\end{example}

\begin{thm}
[Lebesgue-von Neumann-Wold Decomposition]\label{thm:lebesgue-wold}Let
$V=(V_{1},\ldots,V_{n})$ be an isometric $n$-tuple. Then $V$ decomposes
 as
\[
V=V_{u}\oplus V_{a}\oplus V_{s}\oplus V_{d},
\]
where $V_{u}$ is a unilateral $n$-shift, $V_{a}$ is an absolutely
continuous unitary $n$-tuple, $V_{s}$ is a singular unitary $n$-tuple,
and $V_{d}$ is a unitary $n$-tuple of dilation type.\end{thm}
\begin{proof}
The case for $n=1$ follows by the discussion in Section \ref{sec:motivation}.
Thus we can suppose that $n\geq2$. By the Wold decomposition of an
isometric tuple, Proposition \ref{pro:wold-decomp}, we can decompose
$V$ as
\[
V=V_{u}\oplus U,
\]
where $V_{u}$ is a unilateral $n$-shift and $U$ is a unitary $n$-tuple.

By the characterization of an absolutely continuous isometric $n$-tuple
as analytic, Theorem \ref{thm:abs-cont-is-anal}, and the characterization
of a singular isometric $n$-tuple, Corollary \ref{thm:sing-self-adjoint},
an isometric $n$-tuple cannot be both absolutely continuous and singular.
Therefore, we can decompose $U$ as
\[
U=V_{a}\oplus V_{s}\oplus V_{d},
\]
where $V_{a}$ is an absolutely continuous isometric $n$-tuple, $V_{s}$
is a singular isometric $n$-tuple, and $V_{d}$ is of dilation type.
Thus we can further decompose $V$ as
\[
V=V_{u}\oplus V_{a}\oplus V_{s}\oplus V_{d},
\]
as required.
\end{proof}

The next result follows from combining Proposition \ref{pro:struct-dilation-type}
and Theorem \ref{thm:struct-thm-fsa}.
\begin{prop}
\label{lem:struct-alg-dilation-type}Let $V=(V_{1},\ldots,V_{n})$
be an isometric $n$-tuple of dilation type acting on a Hilbert space
$H$. Then there is a projection $P$ and $\alpha\geq1$ and such
that the weakly closed algebra $\mathrm{W}(V)$ generated by $V$
is of the form
\[
\mathrm{W}(V_{1},\ldots,V_{n})=\mathrm{W}^{*}(V)P+P^{\perp}\mathrm{W}(V)P^{\perp},
\]
where $P^{\perp}\mathrm{W}(V_{1},\ldots,V_{n})\mid_{P^{\perp}H}\simeq\mathcal{L}_{n}^{(\alpha)}$.
\end{prop}

The next result follows from the Lebesgue-von Neumann-Wold decomposition
of an isometric tuple, Proposition \ref{pro:wold-decomp}, and Proposition
\ref{lem:struct-alg-dilation-type}.
\begin{thm}
Let $V=(V_{1},\ldots,V_{n})$ be an isometric $n$-tuple acting on
a Hilbert space $H$, and let $V=V_{u}\oplus V_{a}\oplus V_{s}\oplus V_{d}$
be the Lebesgue-von Neumann-Wold decomposition of $V$ as in Theorem
\ref{thm:lebesgue-wold}. Then there is a projection $P$ and $\alpha,\beta\geq0$
such that the weakly closed algebra $\mathrm{W}(V)$ generated by
$V$ is
\[
\mathrm{W}(V)\simeq(\mathcal{L}_{n}(V_{u}\oplus V_{a}))^{(\alpha)}\oplus\mathrm{W}^{*}(V_{s})\oplus\left(\mathrm{W}^{*}(V_{d})P+P^{\perp}\mathrm{W}(V_{d})P^{\perp}\right),
\]
where $P^{\perp}\mathrm{W}(V_{1},\ldots,V_{n})\mid_{P^{\perp}H}\simeq\mathcal{L}_{n}^{(\beta)}$.
The von Neumann algebra $\mathrm{W}^{*}(V_{1},\ldots,V_{n})$ generated
by $V$ is

\[
\mathrm{\mathrm{W}}^{*}(V)\simeq(B(\ell^{2}))^{(\alpha)}\oplus\mathrm{W}^{*}(V_{s})\oplus\mathrm{W}^{*}(V_{d}).
\]
\end{thm}
\begin{acknowledgement*}
The author is grateful to his advisor, Ken Davidson, for his support
and encouragement.\end{acknowledgement*}


\begin{thebibliography}{References}
\bibitem[Arv69]{Arv69}W.B. Arveson, Subalgebras of $\mathrm{C}^{*}$-algebras,
Acta Mathematica 123 (1969), 141--224.

\bibitem[Arv75]{Arv75}W.B. Arveson, Interpolation problems in nest
algebras, Journal of Functional Analysis 20 (1975), No. 3, 208--233.

\bibitem[Ber88]{Ber88}H. Bercovici, Factorization theorems and the
structure of operators on Hilbert space, Annals of Mathematics 128
(1988), No. 2, 399--413.

\bibitem[Ber98]{Ber98}H. Bercovici, Hyper-reflexivity and the factorization
of linear functionals, Journal of Functional Analysis 158 (1998),
No. 1, 242--252.

\bibitem[BFP85]{BFP85}H. Bercovici, C. Foias, C. Pearcy, Dual algebras
with applications to invariant subspaces and dilation theory, CBMS
Regional Conference Series in Mathematics 56 (1985), American Mathematical
Society, Providence.

\bibitem[Bro78]{Bro78}S.W. Brown, Some invariant subspaces for subnormal
operators, Integral Equations and Operator Theory 1 (1978), No. 3,
310--333.

\bibitem[Bun84]{Bun84}J. Bunce, Models for $n$-tuples of non-commuting
operators, Journal of Functional Analysis 57 (1984), No. 1, 21--30.

\bibitem[Cho74]{Cho74}M.D. Choi, A Schwarz inequality for positive
linear maps on $\mathrm{C}^{*}$-algebras, Illinois Journal of Mathematics
18 (1974), No. 4, 565--574.

\bibitem[Con00]{Con00}J.B. Conway, A course in operator theory, Graduate
Studies in Mathematics 21 (2000), American Mathematical Society, Providence.

\bibitem[Dav87]{Dav87}K.R. Davidson,\emph{ }The distance to the analytic
Toeplitz operators, Illinois Journal of Mathematics 31 (1987), No.
2, 265--273.

\bibitem[Dav01]{Dav01}K.R. Davidson, Free semigroup algebras: a survey,
Operator Theory: Advances and Applications 129 (2000), Birkhauser,
Bordeaux.

\bibitem[Dav06]{Dav06}K.R. Davidson, $\mathcal{B}(H)$ is a free
semigroup algebra, Proceedings of the American Mathematical Society
134 (2006), No. 2, 1753--1757.

\bibitem[DKS01]{DKS01}K.R. Davidson, D.W. Kribs, M.E. Shpigel, Isometric
dilations of non-commuting finite rank $n$-tuples, Canadian Journal
of Mathematics 53 (2001), 506--545.

\bibitem[DKP01]{DKP01}K.R. Davidson, E. Katsoulis, D.R. Pitts, The
structure of free semigroup algebras, Journal für die reine und angewandte
Mathematik 533 (2001), 99--125.

\bibitem[DLP05]{DLP05}K.R. Davidson, J. Li, D.R. Pitts, Absolutely
continuous representations and a Kaplansky density theorem for free
semigroup algebras, Journal of Functional Analysis 224 (2005), No.
1, 160--191.

\bibitem[DP98]{DP98}K.R. Davidson, D.R. Pitts, The algebraic structure
of noncommutative analytic Toeplitz algebras, Mathematische Annalen
311 (1998), 275--303.

\bibitem[DP99]{DP99}K.R. Davidson, D.R. Pitts, Invariant subspaces
and hyper-reflexivity for free semigroup algebras, Proceedings of
the London Mathematical Society 78 (1999), No. 2, 401--430.

\bibitem[DY08]{DY08}K.R. Davidson, D. Yang, A note on absolute continuity
in free semigroup algebras, Houston Journal of Mathematics 34 (2008),
283--288.

\bibitem[Fra82]{Fra82}A. Frahzo, Models for non-commuting operators,
Journal of Functional Analysis 48 (1982), No. 1, 1--11.

\bibitem[Ken11]{Ken11} M. Kennedy, Wandering vectors and the reflexivity
of free semigroup algebras, Journal für die reine und angewandte Mathematik
653 (2011), 47--73.

\bibitem[Kri01]{Kri01}D.W. Kribs, Factoring in non-commutative analytic
Toeplitz algebras, Journal of Operator Theory 45 (2001), No. 1, 175--193.

\bibitem[LM78]{LM78}R.I. Loebl, P.S. Muhly, Analyticity and flows
in von Neumann algebras, Journal of Functional Analysis 29 (1978),
No. 2, 214--252.

\bibitem[LS75]{LS75}A.N. Loginov and V.S. Shulman, Hereditary and
intermediate reflexivity of $\mathrm{W}^{*}$-algebras, Izvestiya
Rossiiskoi Akademii Nauk Seriya Matematicheskaya 39 (1975), No. 6,
1260--1273.

\bibitem[MS10]{MS10}P. Muhly, B. Solel, Representations of Hardy
algebras: absolute continuity, intertwiners and superharmonic operators,
preprint (2010), arXiv:1006.1398.

\bibitem[Pop89a]{Pop89a}G. Popescu, Isometric dilations for infinite
sequences of noncommuting operators, Transactions of the American
Mathematical Society 316 (1989), No. 2, 523--536.

\bibitem[Pop89b]{Pop89b}G. Popescu, Multi-analytic operators and
some factorization theorems, Indiana University Mathematics Journal
38 (1989), No. 3, 693--710.

\bibitem[Pop09]{Pop09}G. Popescu,\emph{ }Noncommutative transforms
and free pluriharmonic functions, Advances in Mathematics 220 (2009),
No. 3, 831--893.

\bibitem[Pop91]{Pop91}G. Popescu, von Neumann inequality for $(B(H)^{n})_{1}$,
Mathematica Scandinavica 68 (1991), No. 2, 292--304.

\bibitem[Pop96]{Pop96}G. Popescu, Non-commutative disc algebras and
their representations\emph{,} Proceedings of the American Mathematical
Society 124 (1996), No. 7, 2137--2148.

\bibitem[Read05]{Read05}C.J. Read, A large weak operator closure
for the algebra generated by two isometries, Journal of Operator Theory
54 (2005), No. 2, 305--316.

\bibitem[SF70]{SF70}B. Sz-Nagy, C. Foias, Harmonic analysis of operators
on Hilbert space, Universitext (1970), Springer, North Holland.

\bibitem[Wer52]{Wer52}J. Wermer, On invariant subspaces of normal
operators, Proceedings of the American Mathematical Society 3 (1952),
270--277.\end{thebibliography}
\end{document}